\documentclass[leqno,twoside, 11pt]{amsart}
\usepackage{amssymb, latexsym, esint}
\usepackage{color}

\usepackage{amsthm}
\usepackage{amssymb,amsfonts}
\usepackage[all,arc]{xy}
\usepackage{enumerate}
\usepackage{mathrsfs}
\usepackage{amsmath}
\usepackage{verbatim}

\theoremstyle{plain}
\numberwithin{equation}{section} \numberwithin{figure}{section}

\newtheorem{theorem}{Theorem}[section]
\newtheorem{lemma}[theorem]{Lemma}
\newtheorem{proposition}[theorem]{Proposition}

\newtheorem{definition}[theorem]{Definition}
  \usepackage{epsfig} 
\usepackage{graphicx}  \usepackage{epstopdf}
\theoremstyle{definition}
\newtheorem{remark}[theorem]{Remark}

\numberwithin{equation}{section}
\makeatletter
\@namedef{subjclassname@2020}{\textup{2020} Mathematics Subject Classification}
\makeatother

\begin{document}

\title[Magnetic Kirchhoff equations in Orlicz-Sobolev spaces]{Existence and multiplicity of  solutions to magnetic Kirchhoff equations in Orlicz-Sobolev spaces}

\author{Pablo Ochoa}

\address{P. Ochoa
\hfill\break\indent Facultad de Ingenier\'ia.
 \hfill\break\indent Universidad Nacional de Cuyo-CONICET.
 \hfill\break\indent Mendoza, Argentina.}
 \email{{\tt pablo.ochoa@ingenieria.uncuyo.edu.ar}}

\begin{abstract}
In this paper, we study the existence and multiplicity of weak solutions to a general type of Kirchhoff equations in magnetic fractional Orlicz-Sobolev spaces. Specifically, we appeal to Critical Point Theory to prove the existence of non-trivial solutions under the so-called  Ambrosetti-Rabinowitz condition. We also state the existence of ground-state solutions. Moreover,  multiplicity results which yield the existence of an unbounded sequence of solutions are also provided. Finally, we show existence under a weak-type Ambrosetti-Rabinowitz condition formulated  in the framework of Orlicz spaces.

\end{abstract}



\maketitle
\section{Introduction}
In this work, we consider magnetic Kirchhoff equations involving the fractional operator $\left( -\Delta_g^{A}\right)^{s}$  defined for a  smooth function $u: \mathbb{R}^{N}\to \mathbb{C}$ as
\begin{equation}\label{mag operator}
\left( -\Delta_g^{A}\right)^{s}u(x):=\text{P.V. }\int_{\mathbb{R}^{N}}\dfrac{g\left( |D_s^{A}u(x, y)| \right)}{ |D_s^{A}u(x, y)|} D_s^{A}u(x, y)\dfrac{dy}{|x-y|^{N+s}}\quad x \in \mathbb{R}^{N},
\end{equation}where $s \in (0,1)$, $A: \mathbb{R}^{N}\to \mathbb{R}^{N}$ is a given vector field,  $g=G'$ is the derivative of an N-function, and
$$D_s^{A}u(x, y)=\dfrac{u(x)-e^{i(x-y)\cdot A\left(\frac{x+y}{2}\right)}u(y)}{|x-y|^{s}}.$$Also, P. V. stands for `in principal value'. We refer the reader to Section \ref{preliminares} for details regarding $N$-functions. The operator  $\left( -\Delta_g^{A}\right)^{s}$ is called the magnetic fractional $g$-Laplacian and it arises as the gradient of the non-local magnetic energy functional:

 $$\rho _{s, G}^{A}(u):=\int_{\mathbb{R}^{2N}}G\left( |D_s^{A}u(x, y)|\right)\,d\mu,$$in the sense that
 $$\lim_{t \to 0}\dfrac{\rho_{s, G}^{A}(u+tv)-\rho_{s, G}^{A}(u)}{t}=\mathcal{R}\left(\int_{\mathbb{R}^{2N}}\dfrac{g\left( |D_s^{A}u(x, y)|\right)}{|D_s^{A}u(x, y)|}D_s^{A}u(x, y)\overline{D_s^{A}v(x, y)}\,d\mu \right)=\left\langle \left( -\Delta_g^{A}\right)^{s}u, v \right\rangle.$$Here,  we use the notation
$$d\mu=\dfrac{dx\,dy}{|x-y|^{N}}.$$

Moreover, according to \cite[Theorem 1.4]{FSMag}, the operator $\left( -\Delta_g^{A}\right)^{s}$ may be seen as an approximation of the local operator $-\Delta_{\overline{g}}^{A}$ arising as the gradient of the energy functional
$$\mathcal{I} _{\overline{G}}^{A}(u):=\int_{\mathbb{R}^{N}}\overline{G}\left(|\nabla u-iAu|\right)dx,$$
where $\overline{g}= \overline{G}'$ and $\overline{G}$ is the spherical limit of $G$ (see \cite{FSMag}). For instance, when $G(t)=|t|^{p}$, $p>1$, the spherical limit of $G$ is $\overline{G}(t)=K_{N, p}|t|^{p}$, where $K_{N, p}:= \int_{\mathbb{S}^{N-1}}|x_n|^{p}d\sigma/p$.  More examples are given in \cite{BS}. The approximation is in the following sense: for each $0 <s<1$ and appropriate $f$, let $u_s$ be the unique solution of
$$ \left( -\Delta_g^{A}\right)^{s}=f \quad \text{in }\Omega \subset \mathbb{R}^{N}, \quad u_s=0 \quad \text{in }\mathbb{R}^{N}\setminus \Omega,$$then as $s \to 1$, $u_s \to u$ in $L^{G}(\Omega, \mathbb{C})$, where $u$ satisfies the local equation
$$-\Delta_{\overline{g}}^{A}u=f,\quad \text{in }\Omega \subset \mathbb{R}^{N}, \quad u=0 \quad \text{on }\partial \Omega.$$When $G(t)=|t|^{2}$, the derivative of the corresponding energy functional $\mathcal{I} _{\overline{G}}^{A}$ gives rise to the (real part of the) magnetic Schr\"{o}dinger operator defined as
\begin{equation}\label{local Schrodinger}
-(\nabla -iA)^{2}u(x)=-\Delta u(x) +2iA(x)\cdot\nabla u
+|A(x)|^{2}u+ iu\text{ div}A(x).
\end{equation}
The magnetic Schr\"{o}dinger operator  has been studied extensively in the last decades. We refer the reader to the references \cite{AS}, \cite{CS}, \cite{BD}, \cite{ST}, among many others. 

Observe that when  $A=0$ and $u$ is a real function in \eqref{mag operator}, we recover the now familiar fractional $g$-Laplacian introduced in \cite{BS} and defined as
\begin{equation*}
            (-\Delta_g)^s u (x) = \text{P.V.} 
            \int_{\mathbb{R}^{N}}g
             \Big ( \frac{u(x)-u(y)}{|x-y|^s} \Big ) \frac{dy}
            {|x-y|^{N+s}}
    \end{equation*}for a sufficiently smooth function $u$. Observe that when $g(t)=t$, we get the fractional Laplacian $(-\Delta)^{s}$. Problems with non local diffusion that involve integro-differential operators 
	have been intensively studied in the last years. These nonlocal operators 
	appear when we model different physical situations as anomalous diffusion and 
	quasi-geostrophic flows, turbulence and water waves, molecular dynamics and 
	relativistic quantum mechanics of stars (see \cite{BG,C} and 
	{the references therein}). They also appear  in mathematical finance 
	\cite{A}, elasticity  problems \cite{Si}, 
	phase transition problems  \cite{AB} and crystal dislocation structures
	\cite{T}, among others. We also point out that the operator \eqref{mag operator} constitutes a fractional relativistic generalization of the magnetic Laplacian (see \cite{I1}, \cite{I2} and the nice introduction of \cite{DS}).

	In this work, we consider  global Kirchhoff magnetic equations in Orlicz-Sobolev spaces of the form:\begin{equation}\label{main eq}
M\left( \rho_{s, G}^{A}(u)\right)\left( -\Delta_g^{A}\right)^{s}u + V(x)\dfrac{g(|u|)}{|u|}u = f(x, |u|)u \quad \text{in }\mathbb{R}^{N},
\end{equation}where $V: \mathbb{R}^{N}\to \mathbb{R}$ is the  potential, $f=f(x, t)$ is a given function, and $M: \mathbb{R}_+ \to \mathbb{R}_+$ is the Kirchhoff (tension) function. Kirchhoff \cite{K} introduced a model for a vibrating string given by
$$\rho\frac{\partial^{2}u }{\partial t^{2}}u -\left(\frac{\rho_0}{h}+\frac{E}{2L}\int_0^{L}\big\vert\frac{\partial u }{\partial x} \big\vert\right)\frac{\partial^{2}u }{\partial x^{2}}u=0,$$where $L$ is the length of the string, $h$ the area of the cross-section, E is the Young modulus of the material, $\rho$ is the mass density, and $\rho_0$ the initial tension.  That equation constitutes an extension of the D'Alembert's classical wave equation, since it is taken into account the changes in length of the string under vibration. There have been further generalizations of the Kirchhoff model. In particular, we  quote the nice article \cite{FV}, where the authors studied the existence of non-negative weak solutions for a nonlocal Kirchhoff type problem of the form
$$M\left( \int_{\mathbb{R}^{2N}}\dfrac{|u(x)-u(y)|^{2}}{|x-y|^{N+2s}}dxdy\right)(-\Delta)^{s}u = \lambda f(x, u)+|u|^{2^{*}-2}u.$$In this case, the tension $M$ depends on the length of the string via the fractional norm of $u$.

Some related works to our problem are the following. In \cite{AS}, the authors analysed the existence and multiplicity of solutions to local models like \eqref{main eq} with $M=1$, $A=0$, and involving the Laplacian operator. Related problems in Orlicz spaces have been considered in \cite{FS} and \cite{FSj}, where the main operator is the local g-Laplacian.  Also, in the reference \cite{FV}, the authors considered critical Kirchhoff equations with the fractional Laplacian. There are many other references considering Kirchhoff problems, however, in the next lines, we will focus on the literature more related to the magnetic framework.  Up to our knowledge, the first mathematical contribution to the study of the operator \eqref{mag operator} when $G(t)=|t|^{2}$ was \cite{DS}. There, the authors provide the existence of solutions to minimization problems which in turn implies the solvability of equations like
	$$(-\Delta^{A})^{s}u+u=|u|^{p-2}u \quad \text{in }\mathbb{R}^{N}.$$Afterwards, in \cite{MPSZ}, the authors state existence of global solutions  and   multiplicity results for magnetic  Kirchhoff equations involving the fractional Laplacian via critical point theory.  We also quote the work \cite{FPV}, where they deal with the existence and multiplicity of solutions to boundary value problems in bounded domains driven by the magnetic fractional Laplacian. 
	
	In all the previous work, it is often assumed the Ambrosetti-Rabinowitz condition (AR condition in brief) stated firstly in \cite{AR}, which reads as follows: there exists $\mu >0$ such that
  $$0 < \mu F(x, |u|) \leq f(x, |u|)u^{2}, \quad u \neq 0,$$where $F$ is the primitive of $uf(x, u)$.  Finally, we  point out that in  \cite{P}, the author shows the existence of multiple solutions to magnetic Kirchhoff problems with the p-Laplacian and appealing tow a weak AR condition. 
	
	The functional framework to deal with problems like \eqref{main eq} are the magnetic fractional Orlicz-Sobolev spaces. We refer to the recent work \cite{FSMag}, where the authors provide the basic definitions and properties of these spaces for Lipschitz magnetic fields $A$.  They also give a Bourgain-Brezis-Mironescu type formula and study solutions to the non-local magnetic Laplacian. This study was complemented in the reference \cite{MSV}, where  the Maz'ya-Shaposhnikova formula was provided.

The main results of the present paper are the following. They constitute extensions of the findings of \cite{MPSZ} and \cite{FPV}. We refer to the reader to Section \ref{assumptions sec} for the statements of the assumptions. 
\begin{theorem}\label{existence}
Assume the hypothesis $(Hf)_1-(Hf)_3$, $(HM)_1-(HM)_2$ and $(V)$. Then, there is at least one non trivial weak solution $u \in W_{A, V}^{s, G}(\mathbb{R}^{N}, \mathbb{C})$ to equation \eqref{main eq}.
\end{theorem}

The next result accounts for the existence of ground-state solutions to \eqref{main eq}. We introduce the energy functional associated to \eqref{main eq}:
\begin{equation*}
\mathcal{I}(u):=\mathcal{M}\left( \rho_{s, G}^{A}(u)\right) + \rho_{G, V}(u)-\int_{\mathbb{R}^{N}}F(x, |u|)\,dx, \quad u \in W_{A, V}^{s, G}(\mathbb{R}^{N}, \mathbb{C}),
\end{equation*}where $\mathcal{M}$ is the primitive of $M$.
\begin{theorem}\label{groud state}
Under the assumptions $(Hf)_1-(Hf)_3$, $(HM)_1-(HM)_3$ and $(V)$, there is a weak solution $u_0 \in W_{A, V}^{s, G}(\mathbb{R}^{N}, \mathbb{C})$ to  \eqref{main eq} such that
$$\mathcal{I}(u_0)= \inf\left\lbrace \mathcal{I}(u): \mathcal{I}'(u)=0, u \neq 0\right\rbrace.$$
\end{theorem}

Regarding multiplicity of solutions, we have the next theorem.
\begin{theorem}\label{multiplicity}
If  $(Hf)_1-(Hf)_3$, $(HM)_1-(HM)_2$ and $(V)$ hold, then there exists an unbounded  sequence of weak solutions to \eqref{main eq}.
\end{theorem}

Finally, we mention the following existence result under a weak Ambrosetti-Rabinowitz condition $(Hf)'_3$. However, we will need to impose an additional assumption on the primitive of $f$ (see Section \ref{SEC 6}). 
\begin{theorem}\label{existence2}
Assume that the hypothesis $(Hf)_1, (Hf)_2$, $(Hf)'_3$, $(HM)_1-(HM)_2$ and $(V)$. Then, there is a non trivial weak solution $u \in W_{A, V}^{s, G}(\mathbb{R}^{N}, \mathbb{C})$ to equation \eqref{main eq}.
\end{theorem}

The paper is organized as follows. In Section \ref{preliminares} we give the basic definitions and properties related to N-functions and Orlicz-Sobolev spaces. We also give the main assumptions of the results of the manuscript. In Section \ref{SEC 2}, we provide the proof of Theorem \ref{existence}. The proof of the existence of ground-state solutions is given in Section \ref{SEC 3}. The discussion of the existence of multiple solutions is provided in Section \ref{SEC 4}, and in Section \ref{SEC 6} we give the proof of the existence of solutions with a weak Ambrosetti-Rabinowitz condition. We end the paper with an Appendix, where we provide some further properties of N-functions and Orlicz spaces. 
\section{Preliminaries}\label{preliminares}

\subsection{N-functions and basic properties}
In this section we introduce basic definitions and preliminary results related to Orlicz spaces. We start recalling the definition of an N-function.
	\begin{definition}\label{d2.1}
		A function $G \colon [0, \infty) \rightarrow \mathbb{R}$ is called an N-function if it admits the representation
		$$G(t)= \int _{0} ^{t} g(\tau) d\tau,$$
		where the function $g$ is right-continuous for $t \geq 0$,  positive for $t >0$, non-decreasing and satisfies the conditions
		$$g(0)=0, \quad g(\infty)=\lim_{t \to \infty}g(t)=\infty.$$
		\end{definition}By \cite[Chapter 1]{KR}, an N-function has also the following properties:
		\begin{enumerate}
			\item G is continuous, convex, increasing, even and $G(0) = 0$.
	\item G is super-linear at zero and at infinite, that is $$\lim_{x\rightarrow 0} \dfrac{G(x)}{x}=0$$and
	$$\lim_{x\rightarrow \infty} \dfrac{G(x)}{x}=\infty.$$
	
		\end{enumerate}Indeed, the above conditions serve as an equivalent definition of N-functions.
		
		An important property for N-functions is the following:
		\begin{definition}
		
			 We say that the N-function  G satisfies the $\bigtriangleup_{2}$ condition if  there exists $C > 2$ such that
			\begin{equation*}
				G(2x) \leq C G(x) \,\,\text{~~~for all~~} x \in \mathbb{R}_+.
			\end{equation*}
			\end{definition}
			
			 Examples of functions satisfying the $\bigtriangleup_{2}$ condition are:
    \begin{itemize}
        \item $G(t)= t^{p}$, $t \geq 0$, $p > 1$;
        \item $G(t)= t^{p}(|\log(t)|+1)$, $t \geq 0$, $p > 1$;
        \item $G(t)=(1+|t|)\log(1+|t|) - |t|$;
        \item $G(t)=t^{p}\chi_{(0, 1]}(t) + t^{q}\chi_{(1, \infty)}(t)$, $t \geq 0$, $p, q > 1$.
    \end{itemize}

    
    According to \cite[Theorem 4.1, Chapter 1]{KR}, a necessary and sufficient condition for an N-function to satisfy the $\bigtriangleup_{2}$ condition is that there is $p^{+} > 1$ such that
			\begin{equation}\label{eq p mas}
		\frac{tg(t)}{G(t)} \leq p^{+}, ~~~~~\forall\, t>0.
	\end{equation}In particular, \eqref{eq p mas} implies that an N-function satisfying the $\bigtriangleup_{2}$ condition does not increase more rapidly than a power function (see \eqref{G product} below). Consequently, the following N-function does not satisfies the $\bigtriangleup_{2}$ condition:
	$$G(t) = e^{|t|}-|t|-1.$$ 
	
	Associated to $G$ is  the N-function  complementary to it which is defined as follows:
		\begin{equation}\label{Gcomp}
			\widetilde{G} (t) := \sup \left\lbrace tw-G(w) \colon w>0 \right\rbrace .
		\end{equation}Moreover, the following representation holds for $\widetilde{G}$:
		$$\widetilde{G}(t)=\int_0^{t}g^{-1}(s)\,ds,$$where $g^{-1}$ is the right-continuous inverse of $g$. We recall that the role played by $\widetilde{G}$ is the same as the conjugate exponent functions when $G(t)=t^{p}$, $p >1$.
	
The definition of the complementary function assures that the following Young-type inequality holds
	\begin{equation}\label{2.5}
		at \leq G(t)+\widetilde{G} (a) \text{  for every } a,t \geq 0.
	\end{equation}
	
	We also quote the following useful lemma.
    \begin{lemma}\cite[Lemma 2.9]{BS}\label{G g}
        Let $G$ be an N-function. If $G$ satisfies \eqref{eq p mas} then 
        \[
            \tilde{G}(g(t)) \leq (p^+-1)G(t),
        \]
        where $g=G'$ and $\tilde{G}$ is the complementary function of $G.$
\end{lemma}
	
	By \cite[Theorem 4.3,   Chapter 1]{KR}, a necessary and sufficient condition for the N-function $\widetilde{G} $ complementary to $G$ to satisfy the $\bigtriangleup_{2}$ condition is that there is $p^{-} > 1$ such that
			\begin{equation}
	p^{-} \leq 	\frac{tg(t)}{G(t)}, ~~~~~\forall\, t>0.
	\end{equation}

	From now on, we will assume that the N-function  $G(t)= \int _{0} ^{t} g(\tau) d\tau$  satisfies the following growth behaviour:
	\begin{equation}\label{G1}
		1 < p^{-} \leq \frac{tg(t)}{G(t)} \leq p^{+} < \infty, ~~~~~\forall t>0.
	\end{equation}Hence, $G$ and $\widetilde{G} $ both satisfy the $\bigtriangleup_{2}$ condition. Extending $g$ to the whole of $\mathbb{R}$ by setting $g(t)=-g(-t)$ for $t \leq 0$, we get that \eqref{G1} holds  for all $t \in \mathbb{R}, \, t \neq 0.$ Another useful condition which is assumed in the manuscript  is the following:
	\begin{equation}\label{cG0}
p^{-}-1 \leq \dfrac{tg'(t)}{g(t)}\leq p^{+}-1, ~~~~~\forall t>0.
\end{equation}Again, \eqref{cG0} holds for any $t \in \mathbb{R},\, t \neq 0.$ Observe that \eqref{cG0} implies \eqref{G1}. An example of an N-function satisfying \eqref{cG0} is
$$G(t)=t^p(|\log(t)|+1), \quad p>\tfrac{3+\sqrt{5}}{2}.$$

From \eqref{G1} and  \eqref{cG0}, there follows the next comparison of N-functions with power functions:
   \begin{equation}\label{gg product}
         {\min\left\lbrace a^{p^{-}-1}, a^{p^{+}-1} \right\rbrace g(b) \leq g(ab) \leq  \max\left\lbrace a^{p^{-}-1}, a^{p^{+}-1} \right\rbrace g(b),}
    \end{equation}and
    \begin{equation}\label{G product}
        \min\left\lbrace a^{p^{-}}, a^{p^{+}} \right\rbrace G(b) \leq G(ab) \leq  \max\left\lbrace a^{p^{-}}, a^{p^{+}} \right\rbrace G(b),
    \end{equation}for all $a, b$, where we use the convention $t^{\alpha}= |t|^{\alpha-1}t$. 
    
    \medskip 

Given two $N$-functions $A$ and $B$, we say that $B$ is essentially larger than $A$, denoted by
$A \ll B,$ if for any $c > 0$,
		\begin{equation*}
			\lim_{t \rightarrow \infty} \dfrac{A(ct)}{B(t)}=0.
		\end{equation*}

\subsection{Main assumptions}\label{assumptions sec} In this section, we state the main assumptions for the data of problem \eqref{main eq}.

We start with additional assumptions for the N-function $G$ (recall that it already satisfied  \eqref{cG0}):

$(HG)_1$ Assume $1< p^{-}\leq  p^{+}< \min\left\lbrace (p^{-})^{*}, N \right\rbrace$, where $(p^{-})^{*}= Np^{-}/(N-p^{-})$.

$(HG)_2$ There holds
$$\int_0^{1}\dfrac{G^{-1}(t)}{t^{(N+s)/s}}\,dt < \infty\, \text{ and }\,\int_1^{\infty}\dfrac{G^{-1}(t)}{t^{(N+s)/s}}\,dt=\infty.$$

$(HG)_3$ The function $t \to G(\sqrt{t})$ is convex.

\

In connection with $(HG)_2$, we define the conjugate Orlicz-Sobolev function of $G$ as:
$$(G^{*})^{-1}(t)= \int_0^{t}\dfrac{G^{-1}(w)}{w^{(N+s)/s}}\,dw.$$
The conjugate function will play a role in the embedding results for Orlicz-Sobolev spaces. 

\begin{remark}Observe that when $G(t)=t^{p}$, $p > 1$, $(HG)_2$ holds when $sp < N$. 
\end{remark}

\

Regarding the function $f=f(x, t)$, we will assume:

\

$(Hf)_1$ The function $f: \mathbb{R}^{N}\times \mathbb{R} \to \mathbb{R}$ is a Carath\'eodory function such that 
$$f(x, t)=o(g'(t)) \quad \text{ as }t \to 0.$$

$(Hf)_2$ There exist a constant $c_0 >0$ and an N-function $B$ verifying $B \ll G^{*}$ and $p^{-}_B > p^{+}$, such that
$$|f(x, t)| \leq c_0(1+ b'(t)), \text{ for all }t \geq 0, \,\,\text{where }\,b=B'.$$

$(Hf)_3$ (Ambrosetti-Rabinowitz condition) Let
$$F(x, t):=\int_0^{t}sf(x, s)\,ds.$$Then there exists $\mu >1$ such that
$$0 < \mu F(x, t) \leq f(x, t)t^{2}, \,\text{ for all }t \neq 0$$and
$$\frac{1}{\theta}-\frac{p^{+}}{\mu} >0,$$where $\theta$ is as in $(HM)_1$ below.

\begin{remark}\label{remark} \textit{i)} The assumption $(Hf)_1$ implies that for any $\varepsilon > 0$, there is $\delta >0$ such that
$$|f(x, t)| \leq \varepsilon g'(t)  \quad \text{ if }|t|\leq \delta.$$Now if $t > \delta $, we have
\begin{equation}
\begin{split}
|f(x, t)| \leq c_0(1+ b'(t)) \leq c_0b'(t)\left( 1 + \dfrac{1}{b'(t)}\right) \leq c_0b'(t)\left( 1 + \dfrac{1}{(p_B^{-}-1)\delta^{p_B^{+}-2}}\right). 
\end{split}
\end{equation}As a result,
$$|f(x, t)| \leq \varepsilon g'(t) + C_{\varepsilon}b'(t), \text{ for all }t.$$

\textit{ii)} Observe that the condition $p^{-}_B > p^{+}$ in $(Hf)_2$ implies for $\lambda> 0$,
$$\dfrac{B(t)}{G(\lambda t)} \leq \dfrac{t^{p_B^{-}}}{\lambda^{p^{+}}t^{p^{+}}},$$for all $t$ small enough. Hence,
$$\lim_{t \to 0}\dfrac{B(t)}{G(\lambda t)}=0 \quad \text{ for any }\lambda >0.$$This observation will be needed to get compactness of  embeddings in fractional Orlicz-Sobolev spaces. 
\end{remark}

\

For $M=M(t)$, we will suppose the following:

\

$(HM)_1$ The function $M: \mathbb{R}_+ \to \mathbb{R}_+$ is continuous and for any $\alpha > 0$, there exists $\delta = \delta(\alpha)$ such that
$$M(t) \geq \delta, \,\, \text{for all }t \geq \alpha, \quad t \geq 0.$$

$(HM)_2$ Let 
$$\mathcal{M}(t):= \int_0^{t}M(s)\,ds.$$There exists $\theta \in (1, p_B^{-}/p^{+})$ such that
$$M(t)t \leq \theta \mathcal{M}(t),\quad \text{ for all } t \geq 0.$$

$(HM)_3$ There is $c_0 \in (0, 1)$ such that for all $t \in [0, 1]$ there holds
$$M(t)\geq c_0t^{\theta-1}.$$

\begin{remark}Assumption $(HM)_3$ will be needed to prove the existence of ground-state solutions.  
\end{remark}

The assumptions on the potential $V=V(x)$ are the following:

\

(V) The function $V: \mathbb{R}^{N}\to \mathbb{R}$ is non-negative, continuous in $\mathbb{R}^{N}$  and satisfies
\begin{itemize}
\item $V(x) \geq V_0 >0$ for all $x \in \mathbb{R}^{N}$;
\item For any $M >0$, the set $\left\lbrace x: V(x) < M \right\rbrace$ has finite Lebesgue measure.
\end{itemize}

\subsection{Functional framework}

\subsection{Lebesgue and Orlicz-Sobolev spaces}

We recall the notation
$$d\mu=\dfrac{dx\,dy}{|x-y|^{N}}$$and we introduce
$$D_su(x, y):=\dfrac{u(x)-u(y)}{|x-y|^{s}}.$$
\begin{definition}Let $G$ be an N-function, $s \in (0, 1)$ and $V: \mathbb{R}^{N} \to \mathbb{R}$ be a measurable and bounded function. We define   Lebesgue-Orlicz spaces as follows:
$$L^{G}(\mathbb{R}^{N}, \mathbb{C}):=\left\lbrace u:\mathbb{R}^{N}\to \mathbb{C},\, \text{u is measurable and }\rho_{G}(u) < \infty \right\rbrace$$and
$$L_V^{G}(\mathbb{R}^{N}, \mathbb{C}):=\left\lbrace u:\mathbb{R}^{N}\to \mathbb{C},\, \text{u is measurable and }\rho_{G, V}(u) < \infty \right\rbrace,$$where
$$\rho_{G}(u):=\int_{\mathbb{R}^{N}}G(|u(x)|)\,dx, \quad \text{and }\quad   \rho_{G, V}(u):=\int_{\mathbb{R}^{N}}G(|u(x)|)V(x)\,dx.$$The spaces $L^{G}(\mathbb{R}^{N}, \mathbb{C})$ and $L_V^{G}(\mathbb{R}^{N}, \mathbb{C})$ are equipped with the Luxemburg norms:
$$\|u\|_{G}:=\inf \left\lbrace \lambda >0: \rho_{G}\left( \dfrac{u}{\lambda}\right) \leq 1\right\rbrace$$and $$\|u\|_{G, V}:=\inf \left\lbrace \lambda >0: \rho_{G, V}\left( \dfrac{u}{\lambda}\right) \leq 1\right\rbrace,$$respectively.  The fractional Orlicz-Sobolev spaces are given as$$W^{s, G}(\mathbb{R}^{N}, \mathbb{C})=:\left\lbrace u \in L^{G}(\mathbb{R}^{N}, \mathbb{C}), \rho_{s, G}(u)< \infty \right\rbrace,$$and$$W^{s, G}_V(\mathbb{R}^{N}, \mathbb{C})=:\left\lbrace u \in L_V^{G}(\mathbb{R}^{N}, \mathbb{C}), \rho_{s, G}(u)< \infty \right\rbrace,$$where$$\rho_{s, G}(u)=\int_{\mathbb{R}^{2N}}G(|D_su(x, y)|)\,d\mu.$$The norms in $W^{s, G}(\mathbb{R}^{N}, \mathbb{C})$ and in $W_V^{s, G}(\mathbb{R}^{N}, \mathbb{C})$ are
$$\|u\|_{s, G}:=\|u\|_{G}+[u]_{s, G} \quad \text{and }\quad \|u\|_{s, G}:=\|u\|_{G, V}+[u]_{s, G},$$where$$[u]_{s, G}:=\inf\left\lbrace \lambda >0: \rho_{s, G}\left(\dfrac{u}{\lambda}\right) \leq 1\right\rbrace.$$ \end{definition}For simplicity, when the above spaces are restricted to real-valued functions, we will omit $\mathbb{C}$ in their notation.

The continuity of the embedding $W^{s, G}(\mathbb{R}^{N})$ in $L^{B}(\mathbb{R}^{N})$ is provided in the next theorem from \cite[Theorem 2]{BO}

\begin{theorem}\label{continuity}
Let $G$ be an N-function satisfying \eqref{cG0}. Then, for any N-function $B$ satisfying the $\Delta_2$ condition such that $B \ll G^{*}$, we get that the inclusion
$$W^{s, G}(\mathbb{R}^{N}) \hookrightarrow L^{B}(\mathbb{R}^{N})$$is continuous. Moreover, if $G$ satisfies $(HG)_1$ and $(HG)_2$, then the inclusion
$$W^{s, G}(\mathbb{R}^{N}) \hookrightarrow L^{G^{*}}(\mathbb{R}^{N})$$is continuous.
\end{theorem}


For further reference, we provide the next  Holder's inequality for Orlicz functions with respect to the measure $\mu$.
\begin{lemma}\label{Holder mu} Assume that $U=U(x, y)\in  L_\mu^{G}(\mathbb{R}^{2N})$  and $V=V(x, y)\in  L_\mu^{\tilde{G} }(\mathbb{R}^{2N})$. Then,
$$\int_{\mathbb{R}^{2N}}U(x, y)V(x, y)\,d\mu \leq 2[U]_{s, G}[V]_{s, \tilde{G}}.$$
\end{lemma}
\begin{proof}
By definition,
$$ [U]_{s, G} = \inf\left\lbrace \lambda > 0: \int_{\mathbb{R}^{2N}}G\left(\dfrac{|U(x, y)|}{\lambda} \right)\,d\mu \leq 1 \right\rbrace.$$In particular,
$$\int_{\mathbb{R}^{2N}}G\left(\dfrac{|U(x, y)|}{[U]_{s, G} } \right)\,d\mu \leq 1.$$A similar inequality holds for $V$. Thus, applying Young's inequality \eqref{2.5},
$$\dfrac{U(x, y)}{[U]_{s, G} }.\dfrac{V(x, y)}{[V]_{s, G} } \leq G\left(\dfrac{U(x, y)}{[U]_{s, G} } \right) + \tilde{G}\left(\dfrac{V(x, y)}{[V]_{s, G} } \right),$$and integrating, we obtain
\begin{equation}
\begin{split}
\int_{\mathbb{R}^{2N}}U(x, y)V(x, y)\,d\mu &\leq [U]_{s, G} .[V]_{s, G} \left(\int_{\mathbb{R}^{2N}}G\left( \dfrac{U(x, y)}{[U]_{s, G} }\right) \,d\mu+\int_{\mathbb{R}^{2N}}G\left( \dfrac{V(x, y)}{[V]_{s, G} }\right)\,d\mu \right)\\ & \leq 2[U]_{s, G} .[V]_{s, G}.
\end{split}
\end{equation}
\end{proof}

To close the section, we quote the following relation between modulars and norms.
\begin{lemma}\label{comp norm modular}
        Let $G$ be an N-function satisfying \eqref{G1},  and let 
        $\xi^\pm\colon[0,\infty)$ $\to\mathbb{R}$ be defined as
        \[
            \xi^{-}(t):= 
            \min \big \{  t^{p^{-}}, t^{p^{+}} \big  \} ,
            \quad \text{ and }  \quad
            \xi^{+}(t):=\max \big \{  t^{p^{-}}, t^{p^{+}} \big \} . 
         \] 
         Then
        \begin{itemize}
            \item[(1)]
             $\xi^{-}(\|u\|_{G}) \leq \rho_{G}(u) 
                    \leq  \xi^{+}(\|u\|_{G})$;

   \item[(2)]
    $\xi^{-}([u]_{s, G}) 
                \leq \rho_{s, G}(u) \leq  
                \xi^{+}([u]_{s, G,}).$
        \end{itemize}
    \end{lemma}

\subsection{Magnetic spaces} For a given vector field $A: \mathbb{R}^{N} \to \mathbb{R}^{N}$, we recall
$$D_s^{A}u(x, y)=\dfrac{u(x)-e^{i(x-y)\cdot A\left(\frac{x+y}{2}\right)}u(y)}{|x-y|^{s}}.$$
\begin{definition}Let $G$ be an N-function, $s \in (0, 1)$ and $A: \mathbb{R}^{N} \to \mathbb{R}^{N}$ be a smooth vector field. We define  magnetic  fractional Orlicz-Sobolev spaces as follows:
$$W_{A}^{s, G}(\mathbb{R}^{N}, \mathbb{C}):=\left\lbrace u \in L^{G}(\mathbb{R}^{N}, \mathbb{C}): \rho_{s, G}^{A}(u)< \infty\right\rbrace$$and the weighted space
$$W_{A, V}^{s, G}(\mathbb{R}^{N}, \mathbb{C}):=\left\lbrace u \in L_V^{G}(\mathbb{R}^{N}, \mathbb{C}): \rho_{s, G}^{A}(u)< \infty\right\rbrace,$$where
$$\rho_{s, G}^{A}(u) := \int_{\mathbb{R}^{2N}}G\left( |D_s^{A}u(x, y)|\right)\,d\mu.$$The spaces $W_{A}^{s, G}(\mathbb{R}^{N}, \mathbb{C})$ and $W_{A, V}^{s, G}(\mathbb{R}^{N}, \mathbb{C})$ are Banach spaces when equipped with the Luxemburg norms:
$$\|u\|^{A}_{s, G}:= \|u\|_{G}+[u]_{s, G}^{A}$$
and
$$\|u\|^{A}_{s, G, V}:= \|u\|_{G, V}+[u]_{s, G}^{A}$$where the seminorm $[u]_{s, G}^{A}$ is given by
$$[u]_{s, G}^{A}:=\inf\left\lbrace \lambda >0: \rho_{s, G}^{A}\left(\dfrac{u}{\lambda}\right) \leq 1 \right\rbrace.$$

\end{definition}

The relation between magnetic fractional spaces and $W^{s, G}_V(\mathbb{R}^{N})$ is stated in the next theorem. The proof of the theorem relies on the following diamagnetic inequality:
\begin{equation}\label{diamagnetic ineq}
||u(x)|-|u(y)|| \leq \bigg| u(x)-e^{i(x-y)\cdot A\left( \frac{x+y}{2}\right)}u(y)\bigg|,
\end{equation}which is valid for any finite and measurable magnetic potential field $A$ and any finite and measurable complex function $u$ (for the proof see \cite[Lemma 3.1 and Remark 3.2]{DS}). The meaning of \eqref{diamagnetic ineq} is that removing the magnetic field ($A=0$) allows to decrease the kinetic energy by replacing $u$ with $|u|$ (keeping  $|u|^{2}$ unaltered).

\begin{theorem}\label{embedding continuous}
The embedding $W^{s, G}_{V, A}(\mathbb{R}^{N}, \mathbb{C}) \hookrightarrow W^{s, G}_{V}(\mathbb{R}^{N})$ is continuous in the sense that if $u \in W^{s, G}_{V, A}(\mathbb{R}^{N}, \mathbb{C}) $ then $|u| \in W^{s, G}_{V}(\mathbb{R}^{N})$ and
$$\||u|\|_{s, G, V} \leq \|u\|_{s, G, V}^{A}.$$ 
\end{theorem}
We also point out that since $V_0=\inf_{\mathbb{R}^{N}} V >0$, it follows that
$$\||u|\|_{s, G}\leq C\||u|\|_{s, G, V},$$for some $C>0$. Indeed, assume without loss of generality that $V_0<1$. Let $\lambda >0$ such that$$\int_{\mathbb{R}^{N}}B\left(\dfrac{|u|}{\lambda} \right)V(x)\,dx \leq 1.$$Then,
\begin{equation*}
\begin{split}\int_{\mathbb{R}^{N}}B\left(\dfrac{|u|}{V_0^{-1/p_B^{-}}\lambda} \right)\,dx  &=V_0\left( \dfrac{1}{V_0}\right)^{p_B^{-}/p_B^{-}}\int_{\mathbb{R}^{N}}B\left(\dfrac{|u|}{V_0^{-1/p_B^{-}}\lambda} \right)\,dx \\ & \leq \int_{\mathbb{R}^{N}}B\left(\dfrac{|u|}{\lambda} \right)V(x)\,dx \leq 1 \qquad (\text{by }\eqref{G product}).
\end{split}
\end{equation*}Thus,
$$\||u|\|_{s, G}\leq V_0^{-1/p_B^{-}}\||u|\|_{s, G, V}.$$Hence, the embedding  $W^{s, G}_{V, A}(\mathbb{R}^{N}, \mathbb{C}) \hookrightarrow W^{s, G}(\mathbb{R}^{N})$ is also continuous.


Finally, we mention the following optimal fractional Sobolev inequality from \cite{ACP}. Assume that the  N-function $G$ satisfies:
$$\int^{\infty}\left(\dfrac{t}{G(t)} \right)^{\frac{s}{N-s}}\,dt =\infty$$and
$$\int_0\left(\dfrac{t}{G(t)} \right)^{\frac{s}{N-s}}\,dt <\infty.$$The optimal N-function for the compact embedding is defined as
$$G_{N/s}(t) =G(H^{-1}(t)),\quad t \geq 0$$
where
$$H(t)=\left(\int_0^{t}\left(\frac{\tau}{G(\tau)} \right)^{\frac{s}{N-s}}\,d\tau \right)^{\frac{N-s}{N}}, \quad t \geq 0.$$

The optimal embedding is the next result \cite[Theorem 6.1]{ACP}.
\begin{theorem}\label{compact optimal}
The following embedding is continuous:
$$W^{s, G}(\mathbb{R}^{N}) \hookrightarrow L^{G_{N/s}}(\mathbb{R}^{N}).$$Moreover, there is a constant $C=C(s, N) >0$ such that
$$\|u\|_{G_{N/s}}\leq C[u]_{s, G}.$$
\end{theorem}

The result is optimal in the sense that  if the embedding also holds for another Orlicz function, then the space $L^{G_{N/s}}(\mathbb{R}^{N})$ is continuously embedded into the later. 

\subsection{Compact embeddings}
The following compact embedding follows from \cite[Theorems 1.1 and 1.3]{SCAB}.

\begin{theorem}\label{compact}
Suppose that assumptions $(HG)_1-(HG)_2$ and $(HV)$ hold. Let $B$ be an  N-function satisfying the $\Delta_2$ condition such that
$$\lim_{|t|\to 0}\dfrac{B(t)}{G(\lambda t)}=0, \text{ for all } \lambda> 0,$$and that
$$B\ll G^{*}.$$Then, the embedding
$$W_V^{s, G}(\mathbb{R}^{N}) \hookrightarrow L^{B}_V(\mathbb{R}^{N})$$is compact. Moreover, the embedding $$W_V^{s, G}(\mathbb{R}^{N}) \hookrightarrow L^{G}_V(\mathbb{R}^{N})$$is also compact.
\end{theorem}

We also quote the next compact embedding for magnetic fractional spaces from \cite{FSMag}.

\begin{theorem}\label{compact locally}
Let $s \in (0, 1)$ and $G$ be an N-function satisfying the $\Delta_2$ condition. Then for every bounded sequence in $W^{s, G}_{A}(\mathbb{R}^{N}, \mathbb{C})$ there exist a function $u \in W^{s, G}_{A}(\mathbb{R}^{N}, \mathbb{C})$ and a subsequence $u_{n_k}$  such that 
$$u_{n_k} \to u \text{ in }L^{G}_{loc}(\mathbb{R}^{N}, \mathbb{C}).$$
\end{theorem}

\subsection{Notion of solutions} We start with the definition of weak solutions to \eqref{main eq}.
\begin{definition}
We say that $u \in W_{A, V}^{s, G}(\mathbb{R}^{N}, \mathbb{C})$ is a weak solution of the equation \eqref{main eq} if for any $v \in W_{A, V}^{s, G}(\mathbb{R}^{N}, \mathbb{C})$, there holds
\begin{equation}\label{weak solution}
\begin{split}
&\mathcal{R}\bigg[M\left( \rho_{s, G}^{A}(u)\right)\int_{\mathbb{R}^{2N}}\dfrac{g\left( |D_s^{A}u(x, y)|\right)}{|D_s^{A}u(x, y)|}D_s^{A}u(x, y)\overline{D_s^{A}v(x, y)}\,d\mu \\ & \qquad+ \int_{\mathbb{R}^{N}}\dfrac{g(|u|)}{|u|}u\overline{v}V(x)\,dx \bigg]  = \mathcal{R}\bigg[ \int_{\mathbb{R}^{N}}f(x, |u|)u\overline{v}\,dx\bigg].
\end{split}
\end{equation}
\end{definition}
As a motivation for the definition of weak solutions, consider first the following integral functional $\mathcal{I}: W_{A, V}^{s, G}(\mathbb{R}^{N}, \mathbb{C}) \to \mathbb{R}$ given by
\begin{equation}\label{functional}
\mathcal{I}(u):=\mathcal{M}\left( \rho_{s, G}^{A}(u)\right) + \rho_{G, V}(u)-\int_{\mathbb{R}^{N}}F(x, |u|)\,dx, \quad u \in W_{A, V}^{s, G}(\mathbb{R}^{N}, \mathbb{C}).
\end{equation}Then calculating the Gateaux derivatives, we obtain:
$$\dfrac{d}{d\varepsilon}\mathcal{M}\left( \rho_{s, G}^{A}(u+\varepsilon v)\right)\Big|_{\varepsilon=0}=\mathcal{R}\bigg[M\left( \rho_{s, G}^{A}(u)\right)\int_{\mathbb{R}^{2N}}\dfrac{g\left( |D_s^{A}u(x, y)|\right)}{|D_s^{A}u(x, y)|}D_s^{A}u(x, y)\overline{D_s^{A}v(x, y)}\,d\mu \bigg],$$

$$\dfrac{d}{d\varepsilon}\rho_{G, V}(u+\varepsilon v)\Big|_{\varepsilon=0}=\mathcal{R}\bigg[\int_{\mathbb{R}^{N}}\dfrac{g(|u|)}{|u|}u\overline{v}V(x)\,dx \bigg] $$and
$$\dfrac{d}{d\varepsilon}\left(\int_{\mathbb{R}^{N}}F(x, |u+\varepsilon v|)\,dx\right)\Big|_{\varepsilon=0}=\mathcal{R}\bigg[ \int_{\mathbb{R}^{N}}f(x, |u|)u\overline{v}\,dx\bigg].$$Hence, $u \in  W_{A, V}^{s, G}(\mathbb{R}^{N}, \mathbb{C})$ is a weak solution of $\eqref{main eq}$ if  and only if $u$ is a critical point of $\mathcal{I}$. 

Moreover, observe that all the terms in \eqref{weak solution} are well-defined. Indeed, we first have
\begin{equation}
\begin{split}
&\bigg|\dfrac{g\left( |D_s^{A}u(x, y)|\right)}{|D_s^{A}u(x, y)|}D_s^{A}u(x, y)\overline{D_s^{A}v(x, y)}\bigg| \leq  g\left( |D_s^{A}u(x, y)|\right).|D_s^{A}v(x, y)| \\ & \qquad \leq  \tilde{G}\left( g\left( |D_s^{A}u(x, y)|\right)\right) + G\left( |D_s^{A}v(x, y)|\right) \\ & \qquad \leq G\left( |D_s^{A}u(x, y)|\right)+ G\left( |D_s^{A}v(x, y)|\right) \in L^{1}_\mu(\mathbb{R}^{2N}),
\end{split}
\end{equation}where we have used \eqref{2.5} in the second line and Lemma \ref{G g} in the third line. Regarding the second term in \eqref{weak solution}, it is finite  by Holder's inequality since $g(|u|) \in L_V^{\tilde{G}}(\mathbb{R}^{N})$ and $v \in L_V^{G}(\mathbb{R}^{N})$.  Finally, for $\varepsilon >0$, we get from Remark \ref{remark}
\begin{equation}
\begin{split}
\bigg\vert \int_{\mathbb{R}^{N}}f(x, |u|)u\overline{v}\,dx\bigg\vert &\leq \varepsilon \int_{\mathbb{R}^{N}}g'(|u|)|u||v|\,dx + C_\varepsilon \int_{\mathbb{R}^{N}}b'(|u|)|u||v|\,dx \\ & \leq \varepsilon\int_{\mathbb{R}^{N}}(p^{+}-1)g(|u|)|v|\,dx + C_\varepsilon\int_{\mathbb{R}^{N}}(p_B^{+}-1)b(|u|)|v|\,dx \qquad (\text{by }\eqref{cG0}) \\ & =I_1+I_2.
\end{split}
\end{equation}Since $g(|u|) \in L^{\tilde{G}}(\mathbb{R}^{N})$ by Lemma \ref{G g}, H\"{o}lder's inequality implies that $I_1$ is finite. Regarding $I_2$, we have by Young's inequality \eqref{2.5} that
$$|I_2|\leq C\int_{\mathbb{R}^{N}}\left(\tilde{B}(b(|u|)) + B(|u|) \right)\,dx \leq C \rho_B(u).$$The last term is finite because of the embedding $W^{s, G}_{A, V}(\mathbb{R}^{N}, \mathbb{C}) \hookrightarrow L^{B}(\mathbb{R}^{N})$. 

\section{Existence of solutions: proof of Theorem \ref{existence}}\label{SEC 2}
In this section we will prove the existence of solutions to \eqref{main eq}. It will be enough to prove that $\mathcal{I}$ has critical points. In order to state that, we shall establish first some general properties of the functional $\mathcal{I}$.

Throughout the section we assume  $(HG)_1-(HG)_2$, $(HM)_1-(HM)_2$ and $(Hf)_1-(Hf)_3$.

\begin{lemma}The functional $\mathcal{I}$ belongs to the class $C^{1}(W^{s, G}_{A, V}(\mathbb{R}^{N}, \mathbb{C}), \mathbb{R})$.
\end{lemma}
\begin{proof}
Given $u \in W^{s, G}_{A, V}(\mathbb{R}^{N}, \mathbb{C})$, the Gateaux derivative of $\mathcal{I}$ at $u$ is $$\mathcal{I}'(u)=\mathcal{I}_1'(u)+\mathcal{I}_2'(u)+\mathcal{I}_3'(u),$$where
$$\left\langle \mathcal{I}_1'(u), v \right\rangle = \mathcal{R}\bigg[M\left( \rho_{s, G}^{A}(u)\right)\int_{\mathbb{R}^{2N}}\dfrac{g\left( |D_s^{A}u(x, y)|\right)}{|D_s^{A}u(x, y)|}D_s^{A}u(x, y)\overline{D_s^{A}v(x, y)}\,d\mu \bigg],$$
$$\left\langle \mathcal{I}_2'(u), v \right\rangle =\mathcal{R}\bigg[\int_{\mathbb{R}^{N}}\dfrac{g(|u|)}{|u|}u\overline{v}V(x)\,dx \bigg]$$and
$$\left\langle \mathcal{I}_2'(u), v \right\rangle =\mathcal{R}\bigg[ \int_{\mathbb{R}^{N}}f(x, |u|)u\overline{v}\,dx\bigg].$$Assume that $u_n \to u$ in $W^{s, G}_{A, V}(\mathbb{R}^{N}, \mathbb{C})$.  Then $|u_n| \to |u|$ in $W_V^{s, G}(\mathbb{R}^{N})$. Let $u_{n_k}$ be any subsequence of $u_n$. By Proposition \ref{convergence result}, there is $h \in L^{G}(\mathbb{R}^{N})\cap L_V^{B}(\mathbb{R}^{N})$ and a further subsequence, that we do not relabel, such that
\begin{equation}\label{h}
|u_{n_k}| \leq h \quad a.e.
\end{equation}We shall prove the continuity of $\mathcal{I}'_2$. The argument is similar for $\mathcal{I}'_1$. Suppose that $\|v\|_{s, G, V}^{A}\leq 1$, we have
$$\bigg|\left\langle \mathcal{I}_2'(u_{n_k})-\mathcal{I}_2'(u), v \right\rangle \bigg|=\bigg|\mathcal{R}\bigg[\int_{\mathbb{R}^{N}}\left(\dfrac{g(|u_{n_k}|)}{|u_{n_k}|}u_{n_k}-\dfrac{g(|u|)}{|u|}u\right)\overline{v}V(x)\,dx \bigg] \bigg|.$$Observe that by Lemma \ref{G g},
$$\rho_{\tilde{G}}\left( \dfrac{g(|u_{n_k}|)}{|u_{n_k}|}u_{n_k}\right) \leq C\rho_G(u_{n_k}) < \infty.$$Hence, by H\"{o}lder's inequality, we get
\begin{equation}\label{holder 1}
\bigg|\left\langle \mathcal{I}_2'(u_{n_k})-\mathcal{I}_2'(u), v \right\rangle \bigg| \leq C\bigg\| \bigg\vert\dfrac{g(|u_{n_k}|)}{|u_{n_k}|}u_{n_k}-\dfrac{g(|u|)}{|u|}u\bigg\vert\bigg\|_{\tilde{G}, V}.\||v|\|_{G, V}.
\end{equation}Since
$$\dfrac{g(|u_{n_k}|)}{|u_{n_k}|}u_{n_k}-\dfrac{g(|u|)}{|u|}u \to 0 \quad a.e.$$and 
$$\bigg\vert \dfrac{g(|u_{n_k}|)}{|u_{n_k}|}u_{n_k}-\dfrac{g(|u|)}{|u|}u\bigg\vert \leq g(h)+ g(|u|) \in L^{\tilde{G}}(\mathbb{R}^{N})$$by Lebesgue's Theorem we get
$$\rho_{\tilde{G}}\left( \dfrac{g(|u_{n_k}|)}{|u_{n_k}|}u_{n_k}-\dfrac{g(|u|)}{|u|}u\right) \to 0 \quad \text{ as }n \to \infty.$$Hence, \eqref{holder 1} yields the convergence $\mathcal{I}'_2(u_{n_k}) \to \mathcal{I}'_2(u)$. Since this holds for any subsequence, we prove the convergence of the whole sequence $\mathcal{I}'_2(u_{n}) \to \mathcal{I}'_2(u)$. 

Let us prove the continuity of $\mathcal{I}'_3$. As before, we consider an arbitrary subsequence $u_{n_k}$ and we assume, without loss of generality, that  \eqref{h} holds. Then, for $c >0$ to be chosen later,
\begin{equation}\label{eqq92}
\begin{split}
\bigg|\left\langle \mathcal{I}_3'(u_{n_k})-\mathcal{I}_3'(u), v \right\rangle \bigg| &= \int_{\left\lbrace h < c\right\rbrace }|f(x, |u_{n_k}|)u_{n_k}- f(x, |u|)u||v|\,dx +\\ & +\int_{\left\lbrace h \geq c\right\rbrace }|f(x, |u_{n_k}|)u_{n_k}- f(x, |u|)u||v|\,dx.
\end{split}
\end{equation}We let $A_1=\left\lbrace h < c\right\rbrace$ and $A_2=\left\lbrace h \geq c\right\rbrace$.  Since $G \ll B$ and 
$$\lim_{|t|\to 0}\dfrac{B(t)}{G(t)}=0,$$we have that
$$\lim_{t \to \infty}\dfrac{g(t)}{b(t)}=0 \quad \text{ and }\quad \lim_{|t|\to 0}\dfrac{b(t)}{g(t)}=0.$$Therefore, we may choose $c >0$ such that
\begin{equation}\label{eqq 90}
g(t) \leq b(t) \text{ for }t >c.
\end{equation}Moreover, there is $\delta > 0$, $\delta < c$,  such that
$$b(t) \leq g(t) \text{ for }t \in [0, \delta].$$
Hence, there exists a constant $C>0$ depending on $\delta$ and $c$ such that
\begin{equation}\label{eqq 91}
b(t) \leq g(t)+ Cg(t)\chi_{[\delta,c ]}
\end{equation}for any $t \in [0, c]$, and where $\chi_{[\delta,c ]}$ is the characteristic of the interval $[\delta, c]$.

Over the set $A_1$, we then have by \eqref{eqq 91} and Remark \ref{remark} that
$$|f(x, |u|)u| \leq C\varepsilon g'(|u|)|u| + C_\varepsilon b'(|u|)|u| \leq C_\varepsilon g(h)+C_\varepsilon g(h) \leq C b(h).$$Hence, $|f(x, |u|)|u|| \in L^{\tilde{G}}(A_1)$ and it is dominated by $Cb(h)$. Similarly,$$|f(x, |u_{n_k}|)u_{n_k}| \in L^{\tilde{G}}(A_1).$$ Moreover, over the set $A_2$, we appeal to \eqref{eqq 90} to conclude that $$|f(x, |u|)u|, |f(x, |u_{n_k}|)u_{n_k}| \in L^{\tilde{B}}(A_2)$$and they are majorized by $Cb(h)$. Hence, applying H\"{o}lder's inequality in \eqref{eqq92} gives
\begin{equation}\label{eqq93}
\begin{split}
\bigg|\left\langle \mathcal{I}_3'(u_{n_k})-\mathcal{I}_3'(u), v \right\rangle \bigg| &\leq C\bigg(\||f(x, |u_{n_k}|)u_{n_k}-f(x, |u|)u|\|_{\tilde{G}, A_1}\||v|\|_{G, A_1}\\ & +\||f(x, |u_{n_k}|)u_{n_k}-f(x, |u|)u|\|_{\tilde{B}, A_2}\||v|\|_{B, A_2}\bigg).
\end{split}
\end{equation}
Finally, by Lebesgue's Theorem, it follows that $\mathcal{I}'_3(u_{n_k}) \to \mathcal{I}'_3(u)$. This ends the proof.  
\end{proof}

The following is the definition of the Palais-Smale condition (PS in brief) introduced in \cite{PS}. 

\begin{definition}
We say that the functional  $\mathcal{I}$ satisfies the Palais-Smale condition if any sequence $\left\lbrace u_n \right\rbrace \subset W^{s, G}_{A, V}(\mathbb{R}^{N}, \mathbb{C})$ such that
$\left\lbrace \mathcal{I}(u_n)\right\rbrace$ is bounded and $\mathcal{I}'(u_n)\to 0$ as $n \to \infty$, admits a (strongly) convergent subsequence in $W^{s, G}_{A, V}(\mathbb{R}^{N}, \mathbb{C})$.  
\end{definition}

In the next lemma we prove that $\mathcal{I}$ satisfies the PS-condition.

\begin{lemma}\label{ps condicion}
 The functional $\mathcal{I}$ satisfies the PS-condition in $W^{s, G}_{A, V}(\mathbb{R}^{N}, \mathbb{C})$.  
\end{lemma}

\begin{proof}
Let $\left\lbrace u_n \right\rbrace \subset W^{s, G}_{A, V}(\mathbb{R}^{N}, \mathbb{C})$ such that
$\left\lbrace \mathcal{I}(u_n)\right\rbrace$ is bounded and $\mathcal{I}'(u_n)\to 0$ as $n \to \infty$. We take $C >0$ such that
\begin{equation}\label{C}
|\mathcal{I}(u_n)| \leq C \quad \text{and }\quad |\left\langle \mathcal{I}'(u_n),u_n\right\rangle| \leq C\|u_n\|^{A}_{s, G, V}.\end{equation}First, we suppose that there is $d >0$ such that
\begin{equation}
d=\inf_{n \in \mathbb{N}}[u_n]_{s, G}^{A}.
\end{equation}Then, we have according to \eqref{C} that
\begin{equation}\label{eq 1}
\begin{split}
C+ C\|u_n\|^{A}_{s, G, V} &\geq \mathcal{I}(u_n) -\frac{1}{\mu}\left\langle \mathcal{I}'(u_n), u_n\right\rangle \quad (\text{recall }\mu >1)\\ & \geq \mathcal{M}\left(\rho_{s, G}^{A}(u_n)\right) -\dfrac{p^{+}}{\mu}M\left(\rho_{s, G}^{A}(u_n) \right)\rho_{s, G}^{A}(u_n) +\left(1-\dfrac{p^{+}}{\mu}\right)\rho_{G, V}(u_n) \quad (\text{by }\eqref{G1})\\ & -\frac{1}{\mu}\int_{\mathbb{R}^{N}}\left(\mu F(x, |u_n|)-f(x, |u_n|)|u_n|^{2}\right)\,dx \\ & \geq \frac{1}{\theta}M\left(\rho_{s, G}^{A}(u_n)\right) \rho_{s, G}^{A}(u_n) -\dfrac{p^{+}}{\mu}M\left(\rho_{s, G}^{A}(u_n) \right)\rho_{s, G}^{A}(u_n)\\ & +\left(1-\dfrac{p^{+}}{\mu}\right)\rho_{G, V}(u_n)  \qquad (\text{by } (HM)_2 \text{ and }(Hf)_3) \\ & = \left(\frac{1}{\theta}-\dfrac{p^{+}}{\mu}\right)M\left(\rho_{s, G}^{A}(u_n) \right)\rho_{s, G}^{A}(u_n) +\left(1-\dfrac{p^{+}}{\mu}\right)\rho_{G, V}(u_n) \\&\geq \delta \left(\frac{1}{\theta}-\dfrac{p^{+}}{\mu}\right)\rho_{s, G}^{A}(u_n)  +\left(1-\dfrac{p^{+}}{\mu}\right)\rho_{G, V}(u_n) \quad (\text{by }(HM)_2 \text{ with }\delta=\delta(d)). 
\end{split}
\end{equation}Next, observe that if $\|u_{n_k}\|_{s,G, V}^{A}\to \infty$ for a subsequence, then we may distinguish the following cases. We assume first that $[u_{n_k}]_{s, G}^{A}$ is unbounded and $\|u_{n_k}\|_{G, V} $ is bounded. By \eqref{eq 1}, we get

\begin{equation}\label{eqq 40}
\begin{split}
C+ C\|u_{n_k}\|^{A}_{s, G, V} &\geq\delta \left(\frac{1}{\theta}-\dfrac{p^{+}}{\mu}\right)\left([u_{n_k}]_{s, G}^{A}\right)^{p^{-}} + \left(1-\dfrac{p^{+}}{\mu}\right)\min\left\lbrace \|u_{n_k}\|^{p^{-}}_{G, V}, \|u_{n_k}\|^{p^{+}}_{G, V}\right\rbrace.
\end{split}
\end{equation}Dividing \eqref{eqq 40} by $[u_{n_k}]_{s, G}^{A}$ and letting $k \to \infty$, we get a contradiction. A similar reasoning applies when  $\|u_{n_k}\|_{G, V} $ is unbounded and $ [u_{n_k}]_{s, G}^{A}$ is  bounded.  Finally, if $\|u_{n_k}\|^{A}_{s, G, V}$ and  $\|u_{n_k}\|_{G, V} $ are simultaneously unbounded, we get
from \eqref{eqq 40} that

$$C+ C\|u_{n_k}\|^{A}_{s, G, V} \geq 2^{1-p^{-}}\min \left\lbrace \delta \left(\frac{1}{\theta}-\dfrac{p^{+}}{\mu}\right),  \left(1-\dfrac{p^{+}}{\mu}\right) \right\rbrace \left(\|u_{n_k}\|^{A}_{s, G, V}\right)^{p^{-}},$$which yields a contradiction as $k \to \infty$.

 Hence, we obtain that $\left\lbrace  u_n\right\rbrace$ is bounded in $W^{s, G}_{A, V}(\mathbb{R}^{N}, \mathbb{C})$. So, there is $u \in W^{s, G}_{A, V}(\mathbb{R}^{N}, \mathbb{C})$ so that, up to a subsequence that we do not re label,  $u_n-u \rightharpoonup 0$ in $W^{s, G}_{A, V}(\mathbb{R}^{N}, \mathbb{C})$.  By Theorem \ref{embedding continuous} and Theorems \ref{compact} and \ref{compact locally}, we have that
\begin{equation}\label{eqq 4}
|u_n| \to |u| \text{ in }L_V^{B}(\mathbb{R}^{N}) \quad \text{and }\,\,u_n \to u \text{ in }L^{G}_{loc}(\mathbb{R}^{N}, \mathbb{C}),
\end{equation}with $B$ from the assumption $(Hf)_2$. In particular, $u_n \to u$ a.e. in $\mathbb{R}^{N}$.

Next, define the  linear functionals
$$\left\langle F_1(w), v\right\rangle :=\mathcal{R}\left[ \int_{\mathbb{R}^{2N}}\dfrac{g\left( |D_s^{A}w(x, y)|\right)}{|D_s^{A}w(x, y)|}D_s^{A}w(x, y)\overline{D_s^{A}v(x, y)}\,d\mu\right]$$and
$$\left\langle F_2(w), v\right\rangle :=\mathcal{R}\left[ \int_{\mathbb{R}^{N}}\dfrac{g(|w|)}{|w|}w \overline{v}V(x)\,dx\right],$$for $w, v \in W^{s, G}_{A, V}(\mathbb{R}^{N}, \mathbb{C})$ so that  we may write
$$\left\langle \mathcal{I}'(w), v \right\rangle=M\left( \rho_{s, G}^{A}(w)\right) \left\langle F_1(w), v\right\rangle +\left\langle F_2(w), v\right\rangle -\mathcal{R}\left[\int_{\mathbb{R}^{N}}f(x, |w|)w\overline{v}\,dx \right].$$Observe that
$$\dfrac{g\left( |D_s^{A}w(x, y)|\right)}{|D_s^{A}w(x, y)|}D_s^{A}w(x, y) \in L^{\tilde{G}}_\mu(\mathbb{R}^{2N}) \text{ and }\, \overline{D_s^{A}v(x, y)} \in L^{G}_\mu(\mathbb{R}^{2N}),$$hence by Lemma \ref{Holder mu}, we get that $F_1(w)$ belongs to the dual of $W^{s, G}_{A, V}(\mathbb{R}^{N}, \mathbb{C})$. Therefore,
\begin{equation}\label{eqq 2}
\left\langle F_1(w), u-u_n\right\rangle \to 0 \text{ as }n \to \infty,
\end{equation}for any $w$. Observe that
\begin{equation}\label{eqq 6}
\begin{split}
o(1) &= \left\langle \mathcal{I}'(u_n)-\mathcal{I}'(u), u_n-u \right\rangle \\ & =M\left( \rho_{s, G}^{A}(u_n)\right)  \left\langle F_1(u_n)-F_1(u), u_n-u \right\rangle + \left\langle F_2(u_n)-F_2(u), u_n-u \right\rangle \\& + \left(M\left( \rho_{s, G}^{A}(u_n)\right) -M\left( \rho_{s, G}^{A}(u)\right)  \right)\left\langle F_1(u), u_n-u \right\rangle - \mathcal{R}\left[\int_{\mathbb{R}^{N}}(f(x, |u_n|)u_n -f(x, |u|)u)\overline{u_n-u}\,dx \right].
\end{split}
\end{equation}The term
$$\left(M\left( \rho_{s, G}^{A}(u_n)\right) -M\left( \rho_{s, G}^{A}(u)\right)  \right)\left\langle F_1(u), u_n-u \right\rangle$$converges to $0$ since $M\left( \rho_{s, G}^{A}(u_n)\right)$ remains bounded and by \eqref{eqq 2}. Next, we claim that 
\begin{equation}\label{eqq 5}
\int_{\mathbb{R}^{N}}(f(x, |u_n|)u_n -f(x, |u|)u)\overline{u_n-u}\,dx=o(1) \text{ as }n\to \infty.
\end{equation}Indeed, for $\varepsilon >0$ and by Remark \ref{remark} we have
\begin{equation}
\begin{split}
& \bigg|\int_{\mathbb{R}^{N}}(f(x, |u_n|)u_n -f(x, |u|)u)\overline{(u_n-u)}\,dx\bigg| \leq  \varepsilon\int_{\mathbb{R}^{N}}(g'(|u_n|)|u_n|+g'(|u|)|u|)|u-u_n|\,dx \\ & \qquad + C_\varepsilon\int_{\mathbb{R}^{N}}(b'(|u_n|)|u_n|+b'(|u|)|u|)|u_n-u|\,dx \\ &  \qquad \leq C\varepsilon \int_{\mathbb{R}^{N}}(\tilde{G}( g(|u_n|)+g(|u|))+G(|u_n-u|))\,dx + C_\varepsilon\int_{\mathbb{R}^{N}}(b(|u_n|)+b(|u|))|u_n-u|\,dx\\ & \qquad \leq  C\varepsilon (\rho_G(u_n)+\rho_G(u)+\rho_G(u_n-u))+ C_\varepsilon\|b(|u_n|)+b(|u|)\|_{L^{\tilde{B}}(\mathbb{R}^{N})}.\|u_n-u\|_{L^{B}(\mathbb{R}^{N})}
\end{split}
\end{equation}By Theorem \ref{continuity}, the terms $\rho_G(u_n)+\rho_G(u)+\rho_G(u_n-u)$ and $\|b(|u_n|)+b(|u|)\|_{L^{\tilde{B}}(\mathbb{R}^{N})}$ remains uniformly bounded. Moreover, we know that $\rho_B(u_n)-\rho_B(u) \to 0$ by Holder's inequality and \eqref{eqq 4}.  Hence, the Brezis-Lieb Lemma for Orlicz spaces (see \cite[Lemma 3.4]{FA}, which follows also for complex functions) implies that
$$\|u_n-u\|_{L^{B}(\mathbb{R}^{N})} \to 0,$$which yields as $n \to \infty$ that
$$\bigg|\int_{\mathbb{R}^{N}}(f(x, |u_n|)u_n -f(x, |u|)u)\overline{(u_n-u)}\,dx\bigg| = C\varepsilon + o(1).$$Sending $\varepsilon \to 0$ after taking $n \to \infty$, give \eqref{eqq 5}. 

Next, we analyse in \eqref{eqq 6} the term:
\begin{equation*}
\begin{split}
&\left\langle F_1(u_n)-F_1(u), u_n-u\right\rangle \\& = \mathcal{R}\bigg[ \int_{\mathbb{R}^{N}}\bigg( \dfrac{g(|D_s^{A}u_n(x, y)|)}{|D_s^{A}u_n(x, y)|}D_s^{A}u_n(x, y)-\dfrac{g(|D_s^{A}u(x, y)|)}{|D_s^{A}u(x, y)|}D_s^{A}u(x, y)\bigg)\cdot \overline{D_s^{A}(u_n-u)(x,y)}\,d\mu\bigg]\\ & \,\,
= \int_{\mathbb{R}^{N}}\left(\dfrac{g(|a|)}{|a|}a-\dfrac{g(|b|)}{|b|}b\right)\cdot (a-b)\,d\mu,
\end{split}\end{equation*}where
$$a=(\mathcal{R}(D_s^{A}u_n(x, y)), \,\mathcal{I}(D_s^{A}u_n(x, y))), \quad b=(\mathcal{R}(D_s^{A}u(x, y)), \,\mathcal{I}(D_s^{A}u(x, y))).$$Applying Lemma \ref{inequality g}, we derive

$$\left\langle F_1(u_n)-F_1(u), u_n-u\right\rangle \geq C\rho_G(D_s^{A}u_n-D_s^{A}u).$$Similarly,
$$\left\langle F_2(u_n)-F_2(u), u_n-u\right\rangle \geq C\rho_{G, V}(u_n-u).$$Therefore, the previous inequalities, \eqref{eqq 5} and \eqref{eqq 6} give
$$\rho_{s, G}^{A}(u_n-u)+\rho_{G, V}(u_n-u)=o(1),$$as $n \to \infty$. Hence, $u_n \to u$ strongly in $W^{s, G}_{V, A}(\mathbb{R}^{N}, \mathbb{C})$.

Now, if 
$$\inf_{n \in \mathbb{N}}[u_n]_{s, G}^{A}=0$$then we have two further cases. First, if $0$ is an isolated point of the sequence $[u_n]_{s, G}^{A}$, then we may proceed as before extracting a subsequence. Otherwise, there is a subsequence, not re-label, such that
\begin{equation}\label{eqq 7}
[u_n]_{s, G}^{A} \to 0.
\end{equation}Observe that since $W^{s, G}(\mathbb{R}^{N}) \hookrightarrow  L^{G}(\mathbb{R}^{N})$ is continuous and by the optimal embedding Theorem \ref{compact optimal}, it follows that $L^{G_{N/s}}(\mathbb{R}^{N})\hookrightarrow  L^{G}(\mathbb{R}^{N})$ is continuous and so there is a constant $C>0$ such that
$$\||u_n|\|_{G} \leq C\||u_n|\|_{G_{N/s}} \leq C[|u_n|]_{s, G}\leq C[u_n]_{s, G}^{A}   \to 0 \quad \text{ as } n\to \infty.$$Therefore, $u_n \to 0$ in $W^{s, G}_A(\mathbb{R}^{N}, \mathbb{C})$. By Theorem \ref{compact locally}, up to a subsequence, $u_n \to 0$ a.e. in $\mathbb{R}^{N}$. Moreover, by \eqref{eqq 7} and \eqref{eq 1}, $\|u_n\|_{G, V}$ is uniformly bounded. Hence, again by \eqref{eqq 7}, $u_n$ is bounded in $W^{s, G}_{A, V}$ and so $|u_n|$ is bounded in $W^{s, G}_{V}$. By Theorem \ref{compact}, $|u_n| \to 0$ in $L^{G}_V$. Consequently, $u_n$ converges to $0$ in $W^{s, G}_{A, V}(\mathbb{R}^{N}, \mathbb{C})$. This ends the proof. 
\end{proof}

In the next two lemmas we prove that the functional $\mathcal{I}$ satisfies the geometric  conditions to apply the Mountain Pass Theorem.
\begin{lemma}\label{MP 1}
The exist $r, a> 0$ such that
$$\mathcal{I}(u)\geq a$$for all $u \in W^{s, G}_{V, A}(\mathbb{R}^{N}, \mathbb{C})$ satisfying 
$$\|u\|^{A}_{s, G, V}=r.$$
\end{lemma} 
\begin{proof}
Let $u \in W^{s, G}_{V, A}(\mathbb{R}^{N}, \mathbb{C})$ such that $\|u\|^{A}_{s, G, V}=r$, with $0<r << 1$ to be chosen later. 

First, observe that Remark \ref{remark} implies
\begin{equation}\label{eqq 10}
\begin{split}
\bigg|\int_{\mathbb{R}^{N}}F(x, |u|)\,dx\bigg| &\leq \int_{\mathbb{R}^{N}}\int_0^{|u|}|f(x, s)|s\,ds\,dx \\ & \leq \int_{\mathbb{R}^{N}} \int_0^{|u|} \left(\varepsilon g'(s)  + C_\varepsilon b'(s)\right)s\,ds\,dx \\ & \leq \varepsilon(p^{+}-1) \int_{\mathbb{R}^{N}}\int_0^{|u|}g(s)\,ds\,dx + C_\varepsilon (p_B^{+}-1)\int_{\mathbb{R}^{N}}\int_0^{|u|}b(s)\,ds\,dx \\ & \leq \varepsilon (p^{+}-1)\rho_G(u)+C_\varepsilon (p_B^{+}-1)\rho_B(u) \\ & \leq \dfrac{\varepsilon (p^{+}-1)}{V_0}\rho_{G, V}(u)+C_\varepsilon\dfrac{ (p_B^{+}-1)}{V_0}\rho_{B, V}(u).
\end{split}
\end{equation}Moreover, by Theorems \ref{embedding continuous} and \ref{compact}
\begin{equation*}
\begin{split}
\rho_{B, V}(u) &\leq \max\left\lbrace \||u|\|^{p_B^{+}}_{B,V}, \||u|\|^{p_B^{-}}_{B, V}\right\rbrace \\ & \leq C\max\left\lbrace \||u|\|^{p_B^{+}}_{s, G, V}, \||u|\|^{p_B^{-}}_{s, G, V}\right\rbrace \\& \leq C\max\left\lbrace (\|u\|^{A}_{s, G, V})^{p_B^{+}},  (\|u\|^{A}_{s, G, V})^{p_B^{-}}\right\rbrace = Cr^{p_B^{-}}.
\end{split}
\end{equation*}Hence,
\begin{equation}\label{eqq 11}
\begin{split}
\mathcal{I}(u) &= \mathcal{M}\left(\rho_{s, G}^{A}(u) \right)+ \rho_{G, V}(u) -\int_{\mathbb{R}^{N}}F(x, |u|)\,dx\\ & \geq \mathcal{M}\left(( [u]_{s, G}^{A})^{p^{+}}\right) +\rho_{G, V}(u) -\dfrac{\varepsilon (p^{+}-1)}{V_0}\rho_{G, V}(u)-CC_\varepsilon\dfrac{(p_B^{+}-1)}{V_0}r^{p_B^{-}} \\ & \geq \mathcal{M}(1)([u]_{s, G}^{A})^{p^{+}\theta} + \left(1-\dfrac{\varepsilon (p^{+}-1)}{V_0}\right)\|u\|^{p^{+}\theta}_{G, V}  -CC_\varepsilon\dfrac{(p_B^{+}-1)}{V_0}r^{p_B^{-}} 
\end{split}
\end{equation}where in the last inequality we have used $(HM)_2$ and   $\theta >1$. Let $\varepsilon >0$ so that 
$$\beta:=1-\dfrac{\varepsilon(p^{+}-1)}{V_0} \leq \mathcal{M}(1),$$then \eqref{eqq 11} implies
$$\mathcal{I}(u)  \geq \beta 2^{1-p^{+}\theta}r^{p^{+}\theta} -CC_\varepsilon\dfrac{(p_B^{+}-1)}{V_0}r^{p_B^{-}}.$$Observe that the last term is positive for small $r$ since $p_B^{-}> p^{+}$ and $\theta \in (1, p_B^{-}/p^{+})$ by assumptions $(Hf)_2$ and $(HM)_2$. 
\end{proof}

\begin{lemma}\label{less a} There is $u_0 \in W^{s, G}_{A, V}(\mathbb{R}^{N}, \mathbb{C})$ such that $\|u_0\|_{s, G, V}^{A}> r$ and $\mathcal{I}(u_0) < a$.
\end{lemma}
\begin{proof}

First, observe that from $(Hf)_3$ we get that $\gamma(t)=F(x, t)t^{-\mu}$ is increasing and so $\gamma(t) \geq \gamma(1)$ for $t \geq 1$, which yields
\begin{equation}\label{eqq 12}
F(x, |t|) \geq C|t|^{\mu}.
\end{equation}Moreover, since $f(x, t) = o(g'(t))$, we have $|f(x, t)t| \leq g'(t)t$ for $t \in [0, \delta]$, for some $\delta  \in (0, 1)$. Also, by $(Hf)_2$, there is $C>0$ such that
$$|f(x, t)| \leq 2C, \quad \text{ for all }x \in \mathbb{R}^{N}, \,\, t \in [\delta, 1].$$As a result,
$$f(x, t)t \geq -\left(g'(t)+ 2C \right)t \quad \text{ for all } t\in [0, 1]$$which gives
\begin{equation}\label{eqq 133}
F(x, t)= \int_0^{t}f(x, s)s\,ds \geq -\int_0^{t}g'(s)s\,ds -Ct^{2} \geq -(p^{+}-1)G(t)-Ct^{2}.
\end{equation}Combining \eqref{eqq 12} and \eqref{eqq 133}, we obtain for $t>0$,
\begin{equation}\label{F ineq}
F(x, t) \geq C|t|^{\mu} -C(G(t)+ t^{2}).
\end{equation}Let $u \in C_c^{\infty}(\mathbb{R}^{N}, \mathbb{C})$ with $\|u\|_{s, G, V}^{A} > 1$ and $[u]_{s, G}^{A}=1$. Then for large $t$,
\begin{equation*}
\begin{split}
\mathcal{I}(tu) &= \mathcal{M}\left(\rho_{s, G}^{A}(tu)\right) + \rho_{G, V}(tu) -\int_{\mathbb{R}^{N}}F(x, t|u|)\,dx \\ & \leq \mathcal{M}(1)\left[\rho_{s, G}^{A}(tu) \right]^{\theta}  + \rho_{G, V}(tu)-Ct^{\mu}\|u\|^{\mu}_{L^{\mu}(\mathbb{R}^{N})} + \dfrac{C}{V_0}\rho_{G, V}(tu) + Ct^{2}\|u\|^{2}_{L^{2}(\mathbb{R}^{N})} \\ & \leq  \mathcal{M}(1)t^{p^{+}\theta} + t^{p^{+}}\|u\|_{G, V}^{p^{+}} -Ct^{\mu}\|u\|^{\mu}_{\mu} + Ct^{2}\|u\|^{2}_{2} \to -\infty,
\end{split}
\end{equation*}as $t\to \infty$ since $\mu > p^{+}\theta >2$.
\end{proof}

\subsection*{Proof of Theorem \ref{existence}}

The existence of weak solutions to \eqref{main eq} follows directly from the previous lemmas together with the Mountain Pass Theorem \cite{AR}: 

\begin{theorem}Let $\mathcal{B}$ be a real Banach space. Suppose that $\mathcal{I} \in C^{1}(\mathcal{B})$ satisfies the PS-condition and that:
\begin{itemize}
\item[(i)] $\mathcal{I}(0)=0$;
\item[(ii)] There exist $r, a >0$ such that $\mathcal{I}(u) \geq a$ for all $\|u\|_{\mathcal{B}}=r$;
\item[(iii)] There is $u_0 \in \mathcal{B}$ such that $\|u_0\|_{\mathcal{B}} > r$ and $\mathcal{I}(u_0) < a$.
\end{itemize}Then, $\mathcal{I}$ possesses a critical value $c \geq a$ that can be characterized as
$$c = \inf_{g \in \Gamma}\max_{u \in g([0, 1])}\mathcal{I}(u),$$where
$$\Gamma := \left\lbrace g \in C([0, 1], \mathcal{B}): g(0)=0, \,g(1)=u_0 \right\rbrace.$$
\end{theorem}

Let $u_0 \in W^{s, G}_{A, V}(\mathbb{R}^{N}, \mathbb{C})$ such that $\mathcal{I}(u_0)=c >0$ and $\mathcal{I}'(u_0)=0$. Then $u_0 \neq 0$ and moreover, since $\mathcal{I}$ is even,  $-u_0$ is also a critical point. Therefore, we have shown the existence of at least two non-trivial solutions to \eqref{main eq}.

\section{Existence of ground-state solutions: proof of Theorem \ref{groud state}}\label{SEC 3}

Throughout the section we assume  $(HG)_1-(HG)_2$, $(HM)_1-(HM)_3$ and $(Hf)_1-(Hf)_3$.

\subsection*{Proof of Theorem \ref{groud state}} Define the set of non-trivial critical points
$$\mathcal{C}:=\left\lbrace u \in W^{s, G}_{A, V}(\mathbb{R}^{N}, \mathbb{C})\setminus\left\lbrace 0 \right\rbrace: \mathcal{I}'(u)=0 \right\rbrace.$$Then, by Theorem \ref{existence}, $\mathcal{C} \neq \emptyset$. Let
$$c_m := \inf\left\lbrace \mathcal{I}(u): u \in \mathcal{C} \right\rbrace.$$We first prove that $c_m \geq 0$. Indeed, let $u \in \mathcal{C}$. Since $\mathcal{I}'(u)=0$, we have
\begin{equation}
\begin{split}
\int_{\mathbb{R}^{N}}F(x, |u|)\,dx & \leq \frac{1}{\mu}\int_{\mathbb{R}^{N}}f(x, |u|)|u|^{2}\,dx \\ & \leq \dfrac{p^{+}}{\mu} \left(M(\rho_{s, G}^{A}(u))\rho_{s, G}^{A}(u)+\rho_{G, V}(u) \right). 
\end{split}
\end{equation}Hence,
\begin{equation}\label{eqq 13}
\begin{split}\mathcal{I}(u)& \geq \mathcal{M}(\rho_{s, G}^{A}(u)) + \rho_{G, V}(u) -\dfrac{p^{+}}{\mu} \left(M(\rho_{s, G}^{A}(u))\rho_{s, G}^{A}(u)+\rho_{G, V}(u) \right)\\ & \geq \left(\dfrac{1}{\theta}-\dfrac{p^{+}}{\mu}\right)M(\rho_{s, G}^{A}(u)) \rho_{s, G}^{A}(u) + \left(1- \dfrac{p^{+}}{\mu}\right) \rho_{G, V}(u) \qquad (\text{by }(HM)_2 \text{ and }(Hf)_3)\\& \geq 0.
\end{split}
\end{equation}Thus,  $c_m \geq 0$. 

Moreover,  observe that $c_m$ is attainable. In fact, let $u_n$ be a minimizing sequence in $\mathcal{C}$. Then, $\mathcal{I}'(u_n) = 0$ and $0 \leq  c_m \leq \mathcal{I}(u_n) \leq C$ for all $n$, so  $u_n$ is a PS-sequence. By Lemma \ref{ps condicion}, there is a subsequence of $u_n$, not re-label, and $u \in W^{s, G}_{A, V}(\mathbb{R}^{N}, \mathbb{C})$ such that
$$u_n \to u \quad \text{ in }W^{s, G}_{A, V}(\mathbb{R}^{N}, \mathbb{C}).$$Hence, $\mathcal{I}'(u)=0$ and
$$c_m=\mathcal{I}(u).$$

Next, we prove that $c_m >0$. To get a contradiction, assume that $c_m=0$. Then,  there is a sequence $u_n \in W^{s, G}_{A, V}(\mathbb{R}^{N}, \mathbb{C})$ such that
\begin{equation}
\mathcal{I}'(u_n)=0 \quad \text{ and }\quad \mathcal{I}(u_n) \to 0.
\end{equation}By \eqref{eqq 13}, we get
\begin{equation}\label{eqq 16}
u_n \to 0\quad \text{in }W^{s, G}_{A, V}(\mathbb{R}^{N}, \mathbb{C}).
\end{equation}
Moreover, $\mathcal{I}'(u_n)=0$ implies
\begin{equation}
\begin{split}
&M(\rho_{s, G}^{A}(u_n))\rho_{s, G}^{A}(u_n)+\rho_{G, V}(u_n) \leq \int_{\mathbb{R}^{N}}f(x, |u_n|)|u_n|^{2}\,dx \\ & \qquad \leq \int_{\mathbb{R}^{N}}\left[\varepsilon g'(|u_n|)+C_\varepsilon b'(|u_n|) \right]|u_n|^{2}\,dx\\ & \qquad \leq \dfrac{\varepsilon(p^{+}-1)p^{+}}{V_0}\rho_{G, V}(u_n) + C_\varepsilon\dfrac{(p_B^{+}-1)p_B^{+} }{V_0}\rho_B(u_n)\\ & \qquad \leq \dfrac{\varepsilon(p^{+}-1)p^{+}}{V_0}\rho_{G, V}(u_n) +C.C_\varepsilon(\|u_n\|_{s, G, V}^{A})^{p_B^{-}}.
\end{split}
\end{equation}Consequently,
\begin{equation}\label{eqq 14}
M(\rho_{s, G}^{A}(u_n))\rho_{s, G}^{A}(u_n)+\left(1- \dfrac{\varepsilon(p^{+}-1)p^{+}}{V_0}\right)\rho_{G, V}(u_n) \leq C.C_\varepsilon(\|u_n\|_{s, G, V}^{A})^{p_B^{-}}.
\end{equation}Since $M(t)\geq c_0t^{\theta-1}$ for $t \in [0, 1]$ by $(HM)_3$, we derive from \eqref{eqq 14} that
\begin{equation}\label{eqq 15}
c_0\left([u_n]_{s, G}^{A}\right)^{p^{+}\theta}+\left(1- \dfrac{\varepsilon(p^{+}-1)p^{+}}{V_0}\right)\|u_n\|^{p^{+}}_{G, V} \leq C.C_\varepsilon(\|u_n\|_{s, G, V}^{A})^{p_B^{-}}.
\end{equation}Also, since $\theta p^{+} > p^{+}$ and  $\|u_n\|^{p^{+}}_{G, V} \leq 1$, it follows from \eqref{eqq 15} that
\begin{equation}\label{eqq 16}
c_0\left([u_n]_{s, G}^{A}\right)^{p^{+}\theta}+\left(1- \dfrac{\varepsilon(p^{+}-1)p^{+}}{V_0}\right)\|u_n\|^{p^{+}\theta}_{G, V} \leq C.C_\varepsilon(\|u_n\|_{s, G, V}^{A})^{p_B^{-}}.
\end{equation}Choose now $\varepsilon$ small so that
$$c_0 = 1- \dfrac{\varepsilon(p^{+}-1)p^{+}}{V_0},$$then by \eqref{eqq 16} we have

$$c_0 2^{1-\theta p^{+}}(\|u_n\|_{s, G, V}^{A})^{p^{+}\theta} \leq C.C_\varepsilon (\|u_n\|_{s, G, V}^{A})^{p_B^{-}},$$which is a contradiction since $p^{+}\theta  < p_B^{-}$ by $(HM)_2$ and \eqref{eqq 16} holds. As a result, we have shown that there is a non-trivial ground state solution $u_m$ such that
$$\mathcal{I}(u_m)=c_m >0.$$

\begin{remark}Observe that Theorem \ref{groud state} states that $u=0$ is an isolated critical point of $\mathcal{I}$. 
\end{remark}

\section{Existence of infinitely-many weak solutions: proof of Theorem \ref{multiplicity}}\label{SEC 4}

We will prove Theorem \ref{multiplicity} by appealing to the following result on multiplicity of critical points for $C^{1}$-functionals (see for instance \cite[Theorem 9.12]{AR}).
\begin{theorem}
Suppose that $\mathcal{B}$ is an infinite-dimensional Banach space and suppose that $\mathcal{I}\in C^{1}(\mathcal{B})$ satisfies the PS-condition, is even and $\mathcal{I}(0)=0$. Assume further that $\mathcal{B}=\mathcal{B}^{+}\oplus \mathcal{B}^{-}$, where $\mathcal{B}^{-}$ is finite dimensional, and that the following conditions hold
 \begin{itemize}
 \item[(i)] There exist $a, r >0$ such that for all $u \in \mathcal{B}^{+}$ with $\|u\|=r$, there holds $\mathcal{I}(u) \geq a$.
 \item[(ii)] For any finite dimensional subspace $W \subset \mathcal{B}$, there is $R = R(W)>0$ such that $\mathcal{I}(u) \leq 0$ for all $u \in W$ with $\|u\| \geq R$.
 \end{itemize}Then, $\mathcal{I}$ possesses an unbounded sequence of critical points.
\end{theorem}

\begin{proof}[Proof of Theorem \ref{multiplicity}]We already know that $\mathcal{I}$ satisfies the PS-condition, is even and $\mathcal{I}(0)=0$. We put $\mathcal{B}=W^{s, G}_{A, V}(\mathbb{R}^{N}, \mathbb{C})$ and we let $\mathcal{B}^{+}=\mathcal{B}$ and $\mathcal{B}^{-}=\left\lbrace 0\right\rbrace$. Observe that $(i)$ is satisfied by Lemma \ref{MP 1}.  Hence, it is remains to prove $(ii)$. 

Let $W \subset \mathcal{B}$ be a finite dimensional subspace. Put
$$W=\text{span}\left\lbrace u_1, ..., u_k\right\rbrace.$$In $W$ we define the norm
$$\|u\|= \max\left\lbrace |\alpha_i|: u = \sum_{i=1}^{k}\alpha_iu_i\right\rbrace.$$Since $W$ is finite dimensional, this norm is equivalent to the given norm in $\mathcal{B}$. Assume that $\|u\|=R$, $R>1$ to be chosen later, and that $|\alpha_{i_0}|= \|u\|$ for some $1 \leq i_{i_0}\leq k$. Then 
\begin{equation*}
\begin{split}
\mathcal{I}(u)&\leq \mathcal{M}\left( C_1\sum_{i=1}^{k}\max\left\lbrace |\alpha_i|^{p^{+}},  |\alpha_i|^{p^{-}} \right\rbrace \rho_{s, G}^{A}(u_i)\right) +C_2\sum_{i=1}^{k}\max\left\lbrace |\alpha_i|^{p^{+}}, 
|\alpha_i|^{p^{-}}\right\rbrace \rho_G(u_i) \quad (\text{by }\eqref{G product})\\ & -C_3\bigg(\sum_{i=1}^{k} |\alpha_i|^{\mu}\|u_i\|_{\mu}^{\mu} - \sum_{i=1}^{k} |\alpha_i|^{2}\|u_i\|_{2}^{2}\bigg) \quad (\text{by \eqref{F ineq} and observing }(a+b)^{q} \geq a^{q}+b^{q},\, q \geq 1), \\ & \leq C\left(R^{\theta p^{+}}\left(\sum_{i=1}^{k}\rho_{s, G}^{A}(u_i)\right)^{\theta} +R^{p^{+}}\sum_{i=1}^{k}\rho_G(u_i) + R^{2} \sum_{i=1}^{k}\|u_i\|_{2}^{2}\right) -C_3R^{\mu}\|u_{i_0}\|_{\mu}^{\mu} \leq 0,
\end{split}
\end{equation*}for $R$ large enough, since by assumption $\mu > p^{+}\theta > 2$. Therefore, the conclusion follows.
\end{proof}

\section{Existence of solutions with a weak  Ambrosetti-Rabinowitz condition: proof of Theorem \ref{existence2}}\label{SEC 6}

Throughout this section, we replace hypothesis $(Hf)_3$ with the following weak-type Ambrosetti-Rabinowitz condition:

$(Hf)'_3$ There exists a non negative function $h: \mathbb{R}^{N}\times \mathbb{R}\to \mathbb{R}$ such that
$$\mu F(x,t) -f(x,t)t^{2} \leq h(x, t) \quad \text{ for all }(x, t)\in \mathbb{R}^{N}\times [0, \infty).$$Here, $h$ satisfies the following growth assumption:  
$$|h(x, |t|)| \leq \tilde{B}(|t|) + \phi(x) \quad \text{ for all $x$ and $t$,}$$where $\phi \in L^{1}(\Omega)$, $\phi \geq 0$, and  $\tilde{B}$ is an N-function satisfying \eqref{cG0} and
$$p_{\tilde{B}}^{+} < p^{-}.$$Observe that $W^{s, G}(\mathbb{R}^{N}) \hookrightarrow L^{\tilde{B}}(\mathbb{R}^{N})$ continuously.  Moreover, we assume that
\begin{equation}\label{F lim}
\lim_{|t|\to \infty}\dfrac{F(x, |t|)}{|t|^{p^{+}\theta}}=\infty.
\end{equation}

In order to prove Theorem \ref{existence2}, we first show how to get that $\mathcal{I}$ satisfies the PS-condition. Observe that assumption $(Hf)_3$ was used in the proof of Lemma \ref{ps condicion} in the chain of inequalities \eqref{eq 1}. First, we have
\begin{equation}
\begin{split}
\int_{\mathbb{R}^{N}}h(x, |u_n|)\,dx &\leq \int_{\mathbb{R}^{N}}\tilde{B}(|u_n|)\,dx + \|\phi\|_{1} \\ & = \rho_{\tilde{B}}(u_n) + \|\phi\|_{1}  \\ & \leq \max\left\lbrace \||u_n|\|^{p^{+}_{\tilde{B}}}_{\tilde{B} },  \||u_n|\|^{p^{-}_{\tilde{B}}}_{\tilde{B} }\right\rbrace + \|\phi\|_{1} \\ & \leq C\max\left\lbrace \||u_n|\|^{p^{+}_{\tilde{B}}}_{s, G},  \||u_n|\|^{p^{-}_{\tilde{B}}}_{s, G}\right\rbrace + \|\phi\|_{1} \\ & \leq C\max\left\lbrace (\|u_n\|_{s, G, V}^{A})^{p^{+}_{\tilde{B}}}, (\|u_n\|_{s, G, V}^{A})^{p^{-}_{\tilde{B}}}\right\rbrace + \|\phi\|_{1}.
\end{split}
\end{equation}Hence, under  $(Hf)'_3$, \eqref{eq 1} becomes
\begin{equation}\label{eq 55}
\begin{split}
& C+ C\|u_n\|_{s, G, V}^{A}\geq \delta \left(\frac{1}{\theta}-\frac{p^{+}}{\mu} \right)\rho_{s, G}^{A}(u_n) + \left(1- \frac{p^{+}}{\mu} \right)\rho_{G, V}(u_n)-\frac{1}{\mu}\int_{\mathbb{R}^{N}}h(x, |u_n|)\,dx \\ & \quad \geq  \delta \left(\frac{1}{\theta}-\frac{p^{+}}{\mu} \right)\rho_{s, G}^{A}(u_n) + \left(1- \frac{p^{+}}{\mu} \right)\rho_{G, V}(u_n)\\ & \qquad \quad -C\max\left\lbrace  (\|u_n\|_{s, G, V}^{A})^{p^{+}_{\tilde{B}}}, (\|u_n\|_{s, G, V}^{A})^{p^{-}_{\tilde{B}}}\right\rbrace - \|\phi\|_{1}.
\end{split}
\end{equation}Assume that (a subsequence of) $\|u_n\|_{s, G, V}^{A} $ tends to $\infty$, then \eqref{eq 55} gives
\begin{equation}\label{eq 56}
\begin{split}
&  C(1+  \|u_n\|_{s, G, V}^{A}+  (\|u_n\|_{s, G, V}^{A})^{p^{+}_{\tilde{B}}})\geq \delta \left(\frac{1}{\theta}-\frac{p^{+}}{\mu} \right)\rho_{s, G}^{A}(u_n) + \left(1- \frac{p^{+}}{\mu} \right)\rho_{G, V}(u_n)\\ & \quad \geq C_1\min\left\lbrace([u_n]^{A}_{s, G})^{p^{+}},([u_n]^{A}_{s, G})^{p^{-}}\right\rbrace +C_2\min\left\lbrace \|u_n\|_{G, V}^{p^{+}}, \|u_n\|_{G,V}^{p^{-}} \right\rbrace.
 \end{split}
\end{equation}If $[u_n]^{A}_{s, G} \to \infty$ and $\|u_n\|_{G, V}$ is bounded, dividing both sides of \eqref{eq 56} by  $[u_n]^{A}_{s, G}$ and letting $n \to \infty$ give a contradiction since $p^{+}_{\tilde{B}} < p^{-}$. The argument is similar if $[u_n]^{A}_{s, G}$ is bounded and $\|u_n\|_{G, V}$ is unbounded.  Finally, if both $[u_n]^{A}_{s, G} , \|u_n\|_{G, V} \to \infty$, then \eqref{eq 56} yields
$$ C(1+ \|u_n\|_{s, G, V}^{A} +  (\|u_n\|_{s, G, V}^{A})^{p^{+}_{\tilde{B}}}) \geq C_1([u_n]^{A}_{s, G})^{p^{-}} + C_2\|u_n\|_{G, V}^{p^{-}} \geq C_3(\|u_n\|_{s, G, V}^{A})^{p^{-}},$$and again we get a contradiction since $p^{-}-p^{+}_{\tilde{B}} >0$. 

Finally, assumption \eqref{F lim} is used in the proof of Lemma \ref{less a}. Indeed,  \eqref{eqq 12} becomes now
$$F(x, |t|)\geq C|t|^{p^{+}\theta}$$ for all $t \geq t_0$ and for $C >\mathcal{M}(1)$. Hence, \eqref{F ineq} is now
$$F(x, t) \geq C|t|^{p^{+}\theta} -C(G(t)+t^{2}), \quad t >0.$$Finally,  for $u \in C_c^{\infty}(\mathbb{R}^{N}, \mathbb{C})$ with $\|u\|_{s, G, V}^{A} > 1$ and $[u]_{s, G}^{A}=1$, and for large $t$,
\begin{equation*}
\begin{split}
\mathcal{I}(tu) &\leq \mathcal{M}(1)\left[\rho_{s, G}^{A}(tu) \right]^{\theta}  + \rho_{G, V}(tu)-Ct^{\mu}\|u\|^{p^{+}\theta}_{p^{+}\theta} + \dfrac{C}{V_0}\rho_{G, V}(tu) + Ct^{2}\|u\|^{2}_{L^{2}(\mathbb{R}^{N})} \\ & \leq  \mathcal{M}(1)t^{p^{+}\theta} + t^{p^{+}}\|u\|_{G, V}^{p^{+}} -Ct^{p^{+}\theta}\|u\|^{p^{+}\theta}_{p^{+}\theta} + Ct^{2}\|u\|^{2}_{2} \to -\infty,
\end{split}
\end{equation*}as $t\to \infty$ since $p^{+}\theta >2$ and by the choice of $C$.

\section{Appendix}\label{APP}
\subsection{Inequality for N-functions}
\begin{lemma}\label{inequality g}
Suppose that $G$ is an N-function satisfying \eqref{G1}, \eqref{cG0} and $(HG)_1$. There exists a constant $C>0$ such that for all $a, b \in \mathbb{R}^{N}$ we have
$$\left\langle \dfrac{g(|a|)}{|a|}a-\dfrac{g(|b|)}{|b|}b, a-b \right\rangle \geq CG(|a-b|).$$
\end{lemma}
\begin{proof}We start with some calculations: 
\begin{equation*}
\begin{split}&\left\langle \dfrac{g(|a|)}{|a|}a-\dfrac{g(|b|)}{|b|}b, a-b \right\rangle \\ & \qquad =\left\langle \int_0^{1}\frac{d}{ds}\left(\dfrac{g(|sa+(1-s)b|)}{|sa+(1-s)b|}(sa+(1-s)b)\right)\,ds, a-b \right\rangle \\ & \qquad = \int_0^{1}\dfrac{g'(|sa+(1-s)b|)|sa+(1-s)b|^{2}-g(|sa+(1-s)b|)|sa+(1-s)b|}{|sa+(1-s)b|^{2}}|a-b|^{2}\,ds \\ & \qquad + \int_0^{1}\dfrac{g(|sa+(1-s)b|)}{|sa+(1-s)b|}|a-b|^{2}\,ds \\ & \qquad \geq  (p^{-}-1)\int_0^{1}\dfrac{g(|sa+(1-s)b|)}{|sa+(1-s)b|}|a-b|^{2}\,ds \qquad (\text{ by }\, \eqref{cG0}). 
\end{split}
\end{equation*}As in \cite[Lemma 4.4]{dB}, we consider some cases. If $|a| \geq |a-b|$, then
$$ |sa+(1-s)b| \geq ||a|-(1-s)|a-b||\geq s |a-b|.$$Thus,
\begin{equation*}
\begin{split}
\int_0^{1}\dfrac{g(|sa+(1-s)b|)}{|sa+(1-s)b|}|a-b|^{2}\,ds &\geq \dfrac{1}{p^{-}-1}\int_0^{1}g'(|sa+(1-s)b|)|a-b|^{2}\,ds\\ & \geq \dfrac{1}{p^{-}-1}\int_0^{1}g'(s|a-b|)|a-b|^{2}\,ds \\& =\frac{1}{p^{-}-1}|a-b|g(s|a-b|)\bigg\vert_0^{1} \\ & \geq \dfrac{p^{-}}{p^{-}-1}G(|a-b|).
\end{split}
\end{equation*}On the other hand, if $|a|< |a-b|$, then
\begin{equation*}
\begin{split}
&\int_0^{1}\dfrac{g(|sa+(1-s)b|)}{|sa+(1-s)b|}|a-b|^{2}\,ds \geq p^{-}\int_0^{1}\dfrac{G(|sa+(1-s)b|)}{|sa+(1-s)b|^{2}}|a-b|^{2}\,ds \\ & \geq  \frac{p^{-}}{2}\int_0^{1}G(|sa+(1-s)b|)\,ds \,\quad(\text{since }|sa+(1-s)b|^{2}\leq (2-s)|a-b|^{2})\\ & =\dfrac{p^{-}}{2}\int_0^{1}G\left(\sqrt{|sa+(1-s)b|^{2}} \right)\,ds \\& \geq \dfrac{p^{-}}{2}G\left( \sqrt{\int_0^{1}|sa+(1-s)b|^{2}}ds\right) \\& =\dfrac{p^{-}}{2}G\left( \sqrt{\dfrac{|a|^{2}+\left\langle a, b \right\rangle + |b|^{2} }{3}}ds\right)  \\ & \geq CG(|a-b|),
\end{split}
\end{equation*}where in the last inequality we have used that
$$|a|^{2}+\left\langle a, b \right\rangle + |b|^{2} \geq \frac{1}{4}|a-b|^{2}.$$This finishes the proof. 
\end{proof}
\subsection{A convergence principle in Orlicz spaces}
\begin{proposition}\label{convergence result}
Suppose that $u_n \to u$ in two Orlicz-Lebesgue spaces $L_V^{G}(\mathbb{R}^{N})$ and $L_V^{B}(\mathbb{R}^{N})$. Then, there exists a function $h \in L_V^{G}(\mathbb{R}^{N}) \cap L_V^{B}(\mathbb{R}^{N})$ and a subsequence $u_{n_i}$ such that
$$|u_{n_i}(x)| \leq h(x) \quad a.e..$$
\end{proposition}
\begin{proof}
Since $u_n$ is a Cauchy sequence in both spaces, for each positive integer $i$, there exist $n_i < n_{i+1}$ such that
$$\|u_{n_{i+1}}-u_{n_{i}}\|_{G, V}< 2^{-i} \quad \text{and }\quad\|u_{n_{i+1}}-u_{n_{i}}\|_{B, V}< 2^{-i}.$$Let 
$$f_k(x)=\sum_{i=1}^{k}|u_{n_{i+1}}(x)-u_{n_{i}}(x)|.$$Then,
$$\|f_k\|_{G, V}<1 \quad \text{and }\quad \|f_k\|_{B, V}<1.$$Hence, by Lemma \ref{comp norm modular}
$$\rho_{G, V}(f_k) < 1 \quad \text{and }\quad \rho_{B, V}(f_k) < 1.$$Letting 
$$f(x)= \sum_{i=1}^{\infty}|u_{n_{i+1}}(x)-u_{n_{i}}(x)|,$$we get by Fatou's Lemma, that
$$\rho_{G, V}(f) < 1 \quad \text{and }\quad \rho_{B, V}(f) < 1.$$As a result, $f \in  L_V^{G}(\mathbb{R}^{N}) \cap L_V^{B}(\mathbb{R}^{N})$. Putting
$$h=u_{n_1}+f$$we get $|u_{n_i}|\leq h$ a.e. and $h \in  L_V^{G}(\mathbb{R}^{N}) \cap L_V^{B}(\mathbb{R}^{N}).$ This ends the proof.
\end{proof}

\end{document}